\documentclass[preprint,11pt]{article}
\usepackage{amssymb}
\usepackage{amsfonts}
\usepackage{subfigure}
\usepackage{mathrsfs}
\usepackage{amssymb,amsfonts,amsmath}
\usepackage{graphicx,graphics,epstopdf}
\usepackage{url}
\usepackage[]{algorithm2e}
\usepackage{ctable} 
\graphicspath{{./figures/}}
\usepackage{latexsym}
\usepackage{esint}
\usepackage{times}
\usepackage{multicol,enumerate}
\usepackage{epstopdf}
\usepackage{color}
\usepackage{fullpage}
\usepackage{setspace}
\usepackage[english]{babel}
\usepackage[version=3]{mhchem}
\usepackage{amsmath, amssymb, amsthm}
\usepackage{verbatim, latexsym}
\usepackage{colortbl}
\usepackage{multirow}
\usepackage{float}
\restylefloat{table}
\def\Xint#1{\mathchoice
{\XXint\displaystyle\textstyle{#1}}%
{\XXint\textstyle\scriptstyle{#1}}%
{\XXint\scriptstyle\scriptscriptstyle{#1}}%
{\XXint\scriptscriptstyle\scriptscriptstyle{#1}}%
\!\int}
\def\XXint#1#2#3{{\setbox0=\hbox{$#1{#2#3}{\int}$ }
\vcenter{\hbox{$#2#3$ }}\kern-.6\wd0}}

\def\dashint{\Xint-}
\newcommand{\roma}{\mathrm{I}}
\newcommand{\romb}{\mathrm{II}}
\newcommand{\romc}{\mathrm{III}}

\newcommand{\cellp}{\widehat{Q}_1^{\epsilon}}

\newcommand{\img}{{\bold{i}}}

\newcommand{\mc}[1]{\mathcal{#1}}

\newcommand{\dx}{\,\mathrm{d}x}
\newcommand{\ds}{\,\mathrm{d}s}
\newcommand{\prnt}[1]{\left( #1 \right)}

\newcommand{\norm}[1]{\left\|#1\right\|}
\newcommand{\normHsemi}[2]{\left|#1\right|_{H^{1}\prnt{#2}}}

\newcommand{\normL}[2]{\norm{#1}_{L^2\prnt{#2}}}

\newcommand{\normHp}[3]{\norm{#1}_{H^{#3}\prnt{#2}}}

\newcommand{\Cov}{C_{\mathrm{ov}}}
\newcommand{\Cw}{{\rm C}_{\mathrm{weak}}}
\newcommand{\Ce}{{\rm C}_{\mathrm{est}}}
\newcommand{\Const}[1]{{\rm C}_{\mathrm{#1}}}
\newcommand{\Co}{{\rm C}_{\mathrm{ap}}(k)}
\newcommand{\Cpoin}[1]{{\rm C}_{\mathrm{poin}}(#1)}
\newcommand{\Cpoinn}[2]{{\rm C}^{#2}_{\mathrm{poin}}(#1)}
\newcommand{\locv}[3]{{#1}^{#2}_{\mathrm{#3}}}
\newtheorem{theorem}{Theorem}[section]
\newtheorem{assumption}{Assumption}[section]
\newtheorem{remark}{Remark}[section]
\newtheorem{lemma}{Lemma}[section]

\newtheorem{proposition}{Proposition}[section]
\numberwithin{equation}{section}
\usepackage{soul}
\setstcolor{red}

\title{Wavelet-based Edge Multiscale Finite Element Method for Helmholtz problems in perforated domains} 
\author{Shubin Fu\thanks{Department of Mathematics, The Chinese University of Hong Kong, Hong Kong Special Administrative Region. (\texttt{shubinfu89@gmail.com})} \and Guanglian Li\thanks{Corresponding author. Department of Mathematics, Imperial College London, London SW7 2AZ,
UK. (\texttt{lotusli0707@gmail.com}, \texttt{guanlian.li@imperial.ac.uk}). GL acknowledges the
support from the Royal Society through a Newton international fellowship. GL also acknowledges a Research Impulse grant awarded by Department of Mathematics, Imperial College London.} \and
Richard Craster \thanks{Department of Mathematics, Imperial College London, London SW7 2AZ,
UK. (\texttt{r.craster@imperial.ac.uk})}
\and Sebastien G{u}enneau \thanks {Aix Marseille Univ, CNRS, Centrale
  Marseille, Institut Fresnel, Marseille, France. \newline(\texttt{sebastien.guenneau@fresnel.fr})}
}

\begin{document}
\maketitle
\begin{abstract}
We introduce a new efficient algorithm for Helmholtz problems in
perforated domains with the design of the scheme allowing for possibly large wavenumbers.
Our method is based upon the Wavelet-based Edge Multiscale Finite
Element Method (WEMsFEM) as proposed recently in \cite{fu2018edge}. {For a regular coarse mesh with mesh size $H$,} we
establish $\mathcal{O}(H)$ convergence of this algorithm under the
resolution assumption, and with the level parameter being sufficiently
large. The performance of the algorithm is demonstrated by extensive
2-dimensional numerical tests including those motivated by photonic
crystals.

\end{abstract}
\section{Introduction}
The wave propagation through, and scattering from, complex multiscale
structures is an important area of modern wave physics. The wave manipulation and
control achievable by photonic crystals \cite{joannopoulos08a,zolla12a}, and more
recent metamaterial devices \cite{craster12a,cui10a,maier17a,sarbook}, underlie a host of wave
devices in electromagnetism, optics and acoustics such as optical
fibres, interferometers, mode converters, biosensors, thin-film optics
for reflection control, optical switching and filtering and much more.

The canonical model problem is posed in terms of the Helmholtz
equation  
 in a perforated domain (see Figure \ref{fig:testmodel}):
\begin{equation}\label{eq:model}
\left\{
\begin{aligned}
-(\Delta +k^2)u&=f(x)\qquad&&\text{ in }\Omega^{\epsilon}\\
\frac{\partial u}{\partial n}&=0 \qquad&&\text{ on }  \partial Q_1^{\epsilon}\cap\bar{\Omega}_1\\
\frac{\partial u}{\partial n}-\img k u&=0 \qquad&&\text{ on }  \partial \Omega_2.
\end{aligned}
\right.
\end{equation}
Here, we assume $f\in L^2(\Omega^{\epsilon})$, the wavenumber $k$ is real and positive, $\img$ is the
imaginary unit, $\Omega^{\epsilon}$ is the perforated domain (the
potentially complex structure as a model of a photonic crystal) to be
defined in detail later, $\partial Q_1^{\epsilon}\cap\bar{\Omega}_1$
denotes the interface between the perforations and the perforated
domain $\Omega^{\epsilon}$, and $\partial \Omega_2$ refers to the
outer boundary. Complications arise in resolving the fine
structure in the solution when $k$ is large, at high frequencies, and
this is the regime often of interest in applications.
\begin{figure}[H]
	\centering
	\subfigure[model 1]{
		\includegraphics[trim={4cm 9.5cm 4cm 10cm},clip,width=2.5in]{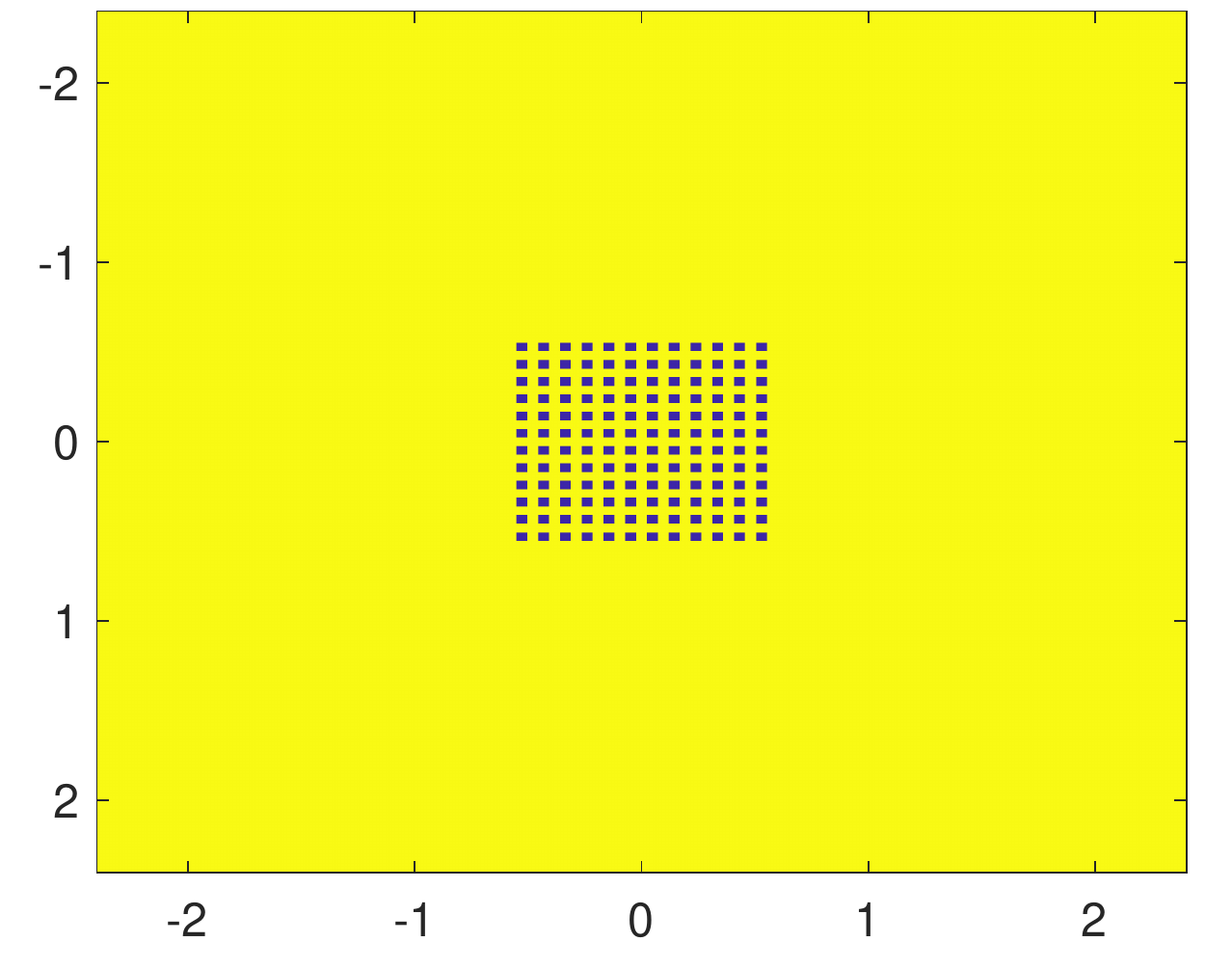}}
	\subfigure[model 2]{
		\includegraphics[trim={4cm 9.5cm 4cm 10cm},clip,width=2.5in]{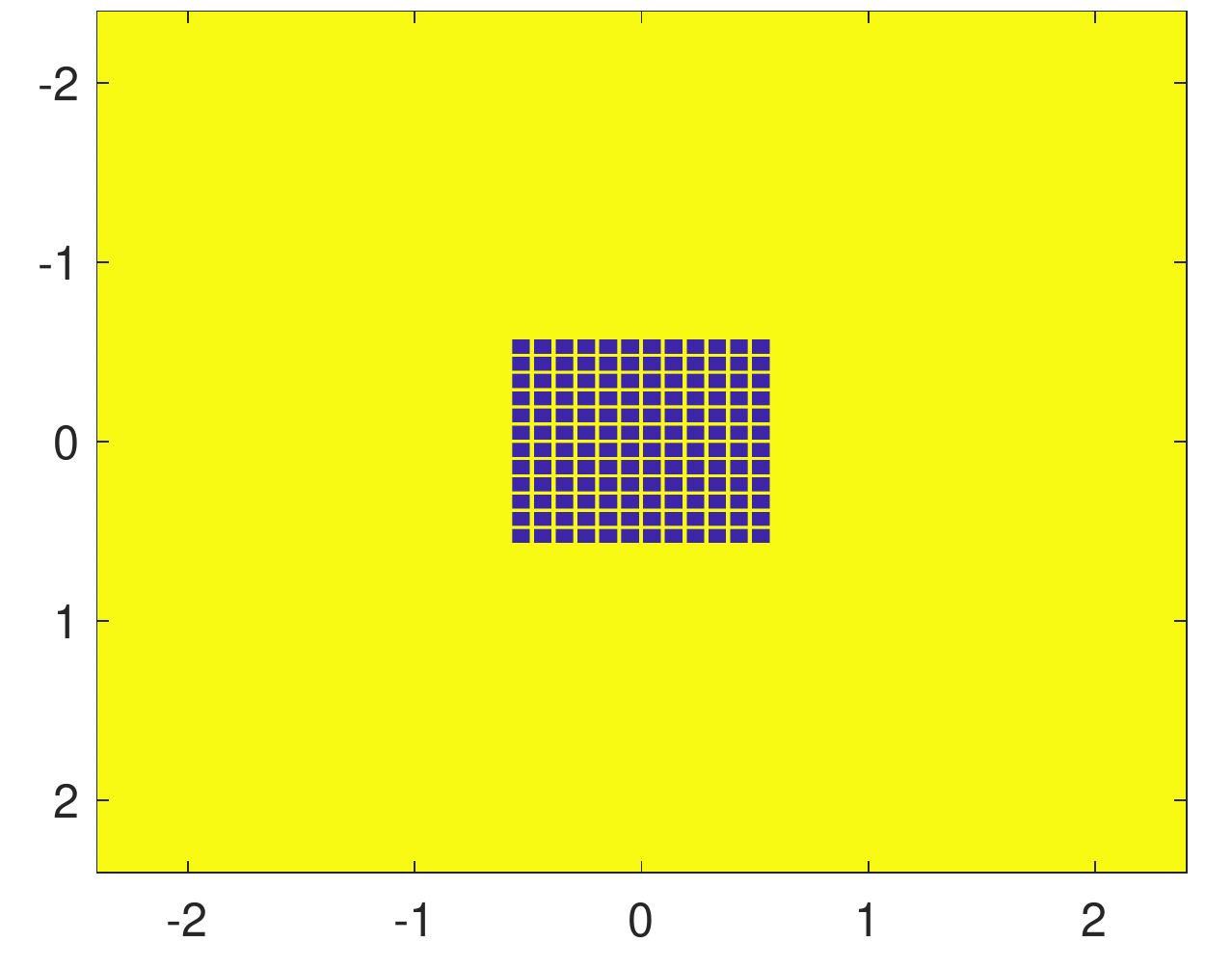}}	
	\caption{Perforated domains. Two models of finite locally periodic
          photonic
          crystals, used later {when they will be given Neumann boundary
          conditions}. Both models have 144 inclusions (in a square
        array $12\times12$), but differ in
          inclusion size, later we introduce forcing and examine the
          wave fields created by the crystal.
           }
	\label{fig:testmodel}
\end{figure}

A key ingredient in any device design, or investigation of a physical effect, is the
accurate, and fast, numerical simulation of the multiscale structures
of interest. This has been the subject of intense and concentrated
research over many years with a variety of techniques employed, the
plane wave expansion methods \cite{joannopoulos08a,johnson01a} are
popular, as are multipole, Rayleigh, methods \cite{zolla12a}, finite
difference time domain (FDTD) \cite{lavrinenko04a} and of
course finite elements \cite{zolla12a}  figure strongly due to
their versatility. Commercial finite element codes such as COMSOL
\cite{comsol} and FDTD such as Lumerical \cite{lumerical}, dominate industry in terms of practicality, but there is a clear
need for more modern implementations of finite-element based numerical
methods in this field; many of the most interesting effects of topical
interest are in three dimensions, such as flat lensing \cite{dubois19a}, antennas, or
involve delicate changes in geometry, as in topological photonics \cite{lu14a}, and may involve many tens, hundreds
or even thousands of cells, each on the micro-scale, forming a
macro-scale object where the wavelengths may be commensurate with the
micro-scale; standard long-wave homogenisation is inappropriate and
standard finite element approaches struggle to cope with the sheer
size of memory required.

Separate from this area of physics, there has been extremely active
research in the modern theory of finite elements and, in particular,
on the development of efficient multiscale methods for practical applications with heterogeneous inseparable multiple scales. Due to this disparity of scales, classical numerical treatments become prohibitively expensive, and even intractable, for many multiscale applications. Nonetheless, motivated by the broad spectrum of practical applications, a large number of multiscale model reduction techniques, e.g., multiscale finite element methods (MsFEMs),
heterogeneous multiscale methods (HMMs), variational multiscale methods, flux norm approach, generalized multiscale finite element methods (GMsFEMs) and localized orthogonal decomposition (LOD), have been proposed in the literature \cite{MR2721592,egh12,MR1979846,MR1455261,MR1660141,li2017error,MR3246801} over
the last few decades. They have achieved great success in the
efficient and accurate simulation of heterogeneous problems; we extend
GMsFEMs
 using wavelets and furthermore investigate their application to the class of wave problems that
encompass photonic crystals.

Designing efficient numerical solvers for Helmholtz equations with
large wavenumbers has also attracted considerable attention over the past few
decades; one of the main challenges is to reduce the so-called wavenumber dependent pollution effect
\cite{MR3585792,MR3123557}. Mitigating the
pollution effect, even for wave propagation through regular structures
with homogeneous physical properties, requires an extremely fine mesh with grid size
depending on the wavenumber $k$, or a very high polynomial degree $p$
within the basis. Consequently this results in an
extremely expensive numerical scheme when the computational domain, or
the wavenumber, is large. Numerical routes based even just around
MsFEMs have not been explored in this context, and one would
anticipate that they might result in efficient numerical solvers for wave
propagation through complex multiscale structures; the WEMsFEM we
develop fits
broadly into the MsFEM framework but with generalisations and
extensions.
It is worthwhile noting that, for multiscale problems, the LOD
approach has been investigated with
\cite{peterseim2017eliminating,peterseim2019computational} having proposed
numerical homogenization to eliminate the pollution
effect for Helmholtz problems in heterogeneous media, however this is
different from perforated domains and so their analysis does not carry
over directly.

In this article we introduce a Wavelet-based Edge Multiscale Finite Element method
(WEMsFEM) for Helmholtz equations in perforated domains, our Algorithm
\ref{algorithm:wavelet}, inspired by the new multiscale algorithm
proposed in \cite{fu2018edge,GL18} for elliptic equations with
heterogeneous coefficients. WEMsFEM takes advantage of the framework of GMsFEM
\cite{egh12}, and utilizes the Partition of Unity Method (PUM)
\cite{Babuska2} as the essential component, and extends these
approaches with additional novel ingredients
that include a rather cheap local solver and provable convergence
rate \cite{fu2018edge}.

The main challenges in Problem \eqref{eq:model} lie in
accurately describing the interfaces, possibly a large number of
perforations, and a large computational domain. The main idea of WEMsFEM
is to utilize wavelets to approximate the solution restricted on the
coarse edges, and then transfer this approximation property to the
interior error estimate. Note that the coarse mesh cannot resolve the
interfaces, nonetheless under the Scale Resolution Assumption
\eqref{ass:resolution} that the coarse mesh grid
$H=\mathcal{O}(k^{-1})$, with $k$ being the wavenumber, we will prove
in Proposition \ref{prop:wavelet-basedconv} that the error of our
proposed multiscale algorithm in the energy norm is of $\mathcal{O}(H)$ given the wavelet parameter $\ell=\mathcal{O}(\log_2 (k C_{\text{ap}}(k)))$ with $C_{\text{ap}}(k)$ being a stability constant defined as in \eqref{res:apriori-dual}.

The remainder of this paper is constructed as follows: We first present, in Section \ref{sec:prelim}, the detailed problem and its basic properties. The Wavelet-based Edge Multiscale Finite Element Method
(WEMsFEM) is introduced in Section \ref{sec:wemsfem} to solve this
problem, and its convergence is analyzed in Section
\ref{sec:convergence}. Furthermore, we present in Section
\ref{sec:numerical} extensive numerical tests to demonstrate the performance of our
proposed algorithm. Finally, we draw together our results for
discussion in Section \ref{sec:conclusion}.


\section{General setting
}\label{sec:prelim}
In this section we present the general setting of the Helmholtz
problem in a perforated domain; we also provide basic properties and
results pertinent to the problem, and an outline of the construction of an ansatz space based on GMsFEM.

We start with the geometric setting of the domain for Problem
\eqref{eq:model}. Let $Q:=[0,1]^d$ be the reference periodicity cell
in $\mathbb{R}^d$ with $d\geq 2$ and we take $Q_0\subset Q$ with infinite
smooth boundary $\partial Q_0$. Denoting $Q_1:=Q\backslash \bar{Q}_0$ as
one unit cell with size $1$,  then the contracted set $\cellp$ is
one cell of the crystal with size $\epsilon$; $\cellp$, and its
$\epsilon$-periodic cloning, $Q_1^{\epsilon}$, are defined as
\[
\widehat{Q}_1^{\epsilon}:=\{x: x/\epsilon\in Q_1\} \qquad \text{and}\qquad
Q_1^{\epsilon}:=\widehat{Q}_1^{\epsilon}+\epsilon \mathbb{Z}^d.
\]

Let $\Omega\subset \mathbb{R}^d$ be a bounded Lipschitz domain, and
also let $\Omega_1\subset \Omega$ and denote by $\Omega_2:=\Omega\backslash \Omega_1$, then the computational domain is
\begin{align}\label{eq:geo}
\Omega^{\epsilon}:= (\Omega_1\cap Q_1^{\epsilon})\cup \Omega_2.
\end{align}


We introduce the complex-valued space $V:=H^{1}(\Omega^{\epsilon};\mathbb{C}):={W^{1,2}(\Omega^{\epsilon};\mathbb{C})}$, equipped with the $k$-weighted norm
\[
\norm{v}_{V}:=\sqrt{\|\nabla v\|_{\Omega^{\epsilon}}^2+k^2\|v\|_{\Omega^{\epsilon}}^2},
\]
and similarly define the complex-valued space $V(D):=H^{1}(D;\mathbb{C})$ also equipped with the $k$-weighted norm
$\norm{\cdot}_{V(D)}$ for all $D\subset \Omega^{\epsilon}$.
Throughout this paper, we denote $(\cdot,\cdot)_{D}$ as the $L^2(D;\mathbb{C})$ inner product for any Lipschitz domain $D$, $\mathrm{Re}\{\cdot\}$, $\mathrm{Im}\{\cdot\}$ and $\bar{\cdot}$ the real part, the imaginary part and the conjugate of a complex value.

The numerical approach we advance uses the weak formulation for
problem \eqref{eq:model} which is to find $u\in V$ such that
\begin{align}\label{eqn:weakform}
a(u,v)=(f,v)_{\Omega^{\epsilon}} \quad \text{for all
} v\in V.
\end{align}
Here, the sesquilinear form $a: V\times V\to \mathbb{C}$ has the form
\begin{align*}
a(v_1,v_2):=\int_{\Omega^{\epsilon}}\nabla v_1\cdot\nabla \bar{v}_2\dx
-k^2\int_{\Omega^{\epsilon}} v_1\cdot\bar{v}_2\dx-\img k\int_{\partial \Omega_2} v_1\cdot\bar{v}_2\ds
\quad \text{ for all } v_1, v_2\in V.
\end{align*}

The following properties of the sesquilinear form $a: V\times V\to \mathbb{C}$ play a critical role, these can be found, e.g., in \cite[Theorem 3.2 and Corollary 3.3]{MR2812565}:
\begin{theorem}[Properties of the sesquilinear form $a: V\times V\to \mathbb{C}$]\label{them:sesquilinear}
The following properties hold:

\begin{itemize}
\item[1.] The sesquilinear form $a: V\times V\to \mathbb{C}$ is bounded: There exists a wavenumber $k$ independent constant, $\Const{b}$, satisfying:
    \[
    |a(v_1,v_2)|\leq \Const{b}\|v_1\|_{V}\|v_2\|_V \qquad\text{ for all }v_1,v_2\in V.
    \]
\item[2.] The following G\aa rding's inequality holds:
\[
\mathrm{Re}\{a(v,v)\}+2k^2\normL{v}{\Omega^{\epsilon}}^2\geq \|v\|_V^2\qquad\text{ for all }v\in V.
\]
\end{itemize}

\end{theorem}
The well-posedness of Problem \eqref{eqn:weakform} can be found, e.g., in \cite[Proposition 8.1.3]{melenk1995generalized}. Furthermore, there exists some constant, $\Co$, that may depend on the wavenumber $k$ and also on the perforated domain $\Omega^{\epsilon}$, such that the unique solution $u\in V$ to Problem \eqref{eqn:weakform} fulfills
\begin{align}\label{res:apriori}
\norm{u}_{V}\leq \Co \normL{f}{\Omega^{\epsilon}}.
\end{align}

Next, we introduce the dual problem to Problem \eqref{eqn:weakform}. For any $w\in L^2(\Omega^{\epsilon};\mathbb{C})$, let $z\in V$ be
\begin{align}\label{eqn:weakformDual}
a(v,z)=(v,w)_{\Omega^{\epsilon}} \quad \text{for all
} v\in V,
\end{align}
 then
\begin{align}\label{res:apriori-dual}
\norm{z}_{V}\leq \Co \normL{w}{\Omega^{\epsilon}}.
\end{align}

\subsection{Ansatz space}
Since the G\aa rding's inequality in Theorem \ref{them:sesquilinear},
combined with the approximation properties of an ansatz space,
implies the quasi-optimality of the conforming Galerkin formulation,
we now introduce the basic construction of the ansatz space.

Let $\mathcal{T}_H$ be a regular partition of the domain $\Omega$ into
finite elements with a mesh size $H$. We refer to
this partition as coarse grids, and the produced elements as the coarse elements. For each coarse element $K\in \mathcal{T}_H$, $K\cap \Omega^{\epsilon}$ is further partitioned
into a union of connected fine grid blocks. The fine-grid partition is denoted by
$\mathcal{T}_h$ with $h$ being its mesh size. Over the fine mesh $\mathcal{T}_h$, let $V_h$ be the conforming piecewise
linear finite element space:
\[
V_h:=\{v\in H^{1}(\Omega^{\epsilon}): V|_{T}\in \mathcal{P}_{1}(T) \text{ for all } T\in \mathcal{T}_h\},
\]
where $\mathcal{P}_1(T)$ denotes the space of linear polynomials on the fine element $T\in \mathcal{T}_h$. Then the fine-scale solution $u_h\in V_h$ satisfies
\begin{align}\label{eqn:weakform_h}
a(u_h,v_h)=(f,v_h)_{\Omega^{\epsilon}} \quad \text{ for all } v_h\in V_h.
\end{align}

The GMsFEM, with which our WEMsFEM shares features and builds from,  aims at solving Problem \eqref{eqn:weakform_h} on the coarse mesh $\mathcal{T}_{H}$
cheaply, whilst simultaneously maintaining a certain accuracy as compared to the fine-scale solution $u_h$. To describe the
GMsFEM, we need some notation: The vertices of $\mathcal{T}_H$
are denoted by $\{O_i\}_{i=1}^{N}$, with $N$ being the total number of coarse nodes.
The coarse neighborhood associated with the node $O_i$ is denoted by
\begin{equation} \label{neighborhood}
\omega_i:=\bigcup\{ K_j\in\mathcal{T}_H: ~~~ O_i\in \overline{K}_j\}.
\end{equation}
The overlap constant $\Cov$ is defined by
\begin{align}\label{eq:overlap}
\Cov:=\max\limits_{K\in \mathcal{T}_{H}}\#\{O_i: K\subset\omega_i \text{ for } i=1,2,\cdots,N\}.
\end{align}
We refer to Figure~\ref{schematic} for an illustration of neighborhoods and elements subordinated to the coarse
discretization $\mathcal{T}_H$. Throughout, we use $\omega_i$ to denote a coarse neighborhood.

\begin{figure}[htb]
  \centering
  \includegraphics[width=0.65\textwidth]{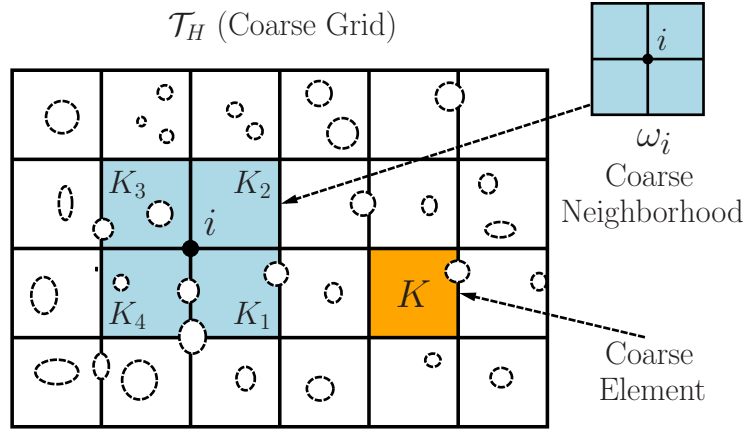}
  \caption{Illustration of a coarse neighborhood and coarse element with an overlapping constant $\Cov=4$ in a perforated domain. Here, the dashed lines denote the interfaces between the computational domain and the perforations.}
  \label{schematic}
\end{figure}

Next, we outline the GMsFEM with a conforming Galerkin (CG) formulation. We denote by $\omega_i$
the support of the multiscale basis functions. These basis functions are denoted by $\psi_j^{\omega_i}$ for
$j=1,\cdots,\ell_i$ for some $\ell_i\in \mathbb{N}_{+}$, which is the number of local basis functions associated with $\omega_i$. Throughout,
the superscript $i$ denotes the $i$-th coarse node or coarse neighborhood $\omega_i$.
Generally, the GMsFEM utilizes multiple basis functions per coarse neighborhood $\omega_i$,
and the index $j$ represents the numbering of these basis functions.
In turn, the CG multiscale solution $u_{\text{ms}}$ is sought as $u_{\text{ms}}=\sum_{i,j} c_{j}^i \psi_{j}^{\omega_i}$.
Once the basis functions $\psi_j^{\omega_i}$ are identified, the CG global coupling is given through the variational form
\begin{equation}\label{eq:globalG}
a(u_{\text{ms}},v)=(f,v)_{\Omega^{\epsilon}} \quad \text{for all} \, \, v\in
V_{\text{ms}}.
\end{equation}
Here, $V_{\text{ms}}:=\text{span}\{\psi_{j}^{\omega_i}: i=1,2,\ldots,N, j=1,2,\ldots,\ell_i\}$ denotes the ansatz space.

\section{WEMsFEM}\label{sec:wemsfem}
We now present our main multiscale method to efficiently solve Problem
\eqref{eq:model} and note that we are especially interested in the
cases where the size of the cell $\epsilon$ is very small and the
wavenumber $k$ is large. Our method is based on the edge multiscale
method proposed in \cite{fu2018edge}. The main idea is to utilize the
wavelets to approximate $u|_{\partial\omega_i}$ and then transfer this
approximation property into the error over the global domain
$\Omega^{\epsilon}$. For the completeness of the presentation, we
introduce the wavelets in the following section, the details of which
are also found, e.g., in \cite{fu2018edge}.

\subsection{Haar wavelet}
Let the level parameter and the mesh size be $\ell$ and
$h_{\ell}:=2^{-\ell}$ with $\ell\in \mathbb{N}$, respectively, then the grid points on level $\ell$ are
\[
x_{\ell,j}=j{\times} h_{\ell},\qquad 0\leq j\leq 2^{\ell}.
\]
Let the scaling function $\phi(x)$ and the mother wavelet $\psi(x)$ be given by
\begin{equation*}
\phi(x)=
\left\{
\begin{aligned}
&1, &&\text{ if } 0\leq x\leq 1,\\
&0, &&\text{ otherwise,}
\end{aligned}
\right.
\qquad
\psi(x)=
\left\{
\begin{aligned}
&1, &&\text{ if } 0\leq x\leq 1/2,\\
&-1, &&\text{ if }1/2< x\leq 1,\\
&0, &&\text{ otherwise}.
\end{aligned}
\right.
\end{equation*}
By means of dilation and translation, the mother wavelet $\psi(x)$ can
result in an orthogonal decomposition of the space $L^2(I)$ with $I:=[0,1]$. To this end, we can define the basis functions on level $\ell\geq 1$ by
\begin{align*}
\psi_{\ell,j}(x):=2^{\frac{\ell-1}{2}}\psi(2^{\ell-1}x-j) \quad \text{ for all }\quad 0\leq j\leq 2^{\ell-1}-1.
\end{align*}
The subspace of level $\ell$ is
\begin{equation*}
W_{\ell}:=
\left\{
\begin{aligned}
&\text{span}\{\phi\} &&\text{ for }\ell =0\\
&\text{span}\{\psi_{\ell,j}:\quad 0\leq j\leq 2^{\ell-1}-1\}&&\text{ for }\ell\geq 1.
\end{aligned}
\right.
\end{equation*}
 and we note that subspace $W_{\ell}$ is orthogonal to $W_{\ell'}$ in
 $L^2(I)$ for any two different levels $\ell\neq \ell'$. We denote the
 subspace in $L^2(I)$, up to level $\ell$, by $V_{\ell}^{\roma}$  defined by
\[
V_{\ell}:=\oplus_{m\leq\ell}W_{m}.
\]
The orthogonality of the subspaces $W_{\ell}$ on different levels
leads to the relation
\[
V_{\ell+1}=V_{\ell}\oplus_{L^2(I)} W_{\ell+1}
\]
and, consequently, yields the hierarchical structure of the subspace $V_{\ell}$, namely,
\[
V_{0}\subset V_{1}\subset \cdots\subset V_{\ell}\subset V_{\ell+1}\cdots
\]
Furthermore, the following orthogonal decomposition of the space $L^2(I)$ holds
\[
L^2(I)=\oplus_{\ell}W_{\ell}.
\]
Note that one can derive the hierarchical decomposition of the space
$L^2(I^d)$ for $d>1$ by means of the tensor product. The following approximation property holds \cite[Proposition 3.1]{fu2018edge}:
\begin{proposition}[Approximation properties of the hierarchical space $V_{\ell}$ in \cite{fu2018edge}]\label{prop:approx-wavelets}
Let $P_{\ell}$ be $L^2(I)$-orthogonal projection onto $V_{\ell}$ for each level $\ell\geq 0$ and let $s>0$. Then there holds
\begin{align*}
P_{\ell+1} v&=P_{\ell}v+\sum\limits_{j=0}^{2^{\ell}-1}(v,\psi_{\ell+1,j})_{I}\psi_{\ell+1,j}&\text{for all }v\in L^2(I)\\
\|v-P_{\ell}v\|_{L^2(I)}&\lesssim 2^{-s\ell}|v|_{H^s(I)} &\text{for all }v\in H^s(I).
\end{align*}
\end{proposition}

\subsection{The method}
We propose our main multiscale algorithm in this section, specifically
the corresponding multiscale basis functions are defined locally on
each coarse neighborhood independently, and thereby are calculated in
parallel.  Essentially, we apply wavelets to approximate the solution
restricted on each coarse edge. To obtain conforming global basis
functions, we utilize the Partition of Unity finite element method
\cite{EFENDIEV2011937,melenk1996partition}; the main idea is to
seek local multiscale basis functions in each coarse neighborhood,
 having certain approximation properties to the exact solution restricted on each coarse neighborhood, and use the fact that the global multiscale basis functions, obtained from those local multiscale basis functions by the partition of unity functions, inherit these approximation properties.

To this end, we begin with an initial coarse space $V^{\text{init}}_0
= \text{span}\{ \chi_i \}_{i=1}^{N}$, with
 the $\chi_i$  as the standard multiscale basis functions on each
 coarse element $K\in \mathcal{T}_{H}$ defined via
\begin{alignat}{2} \label{pou}
-(\Delta +k^2)\chi_i &= 0  &&\quad\text{ in }\;\;\Omega^{\epsilon}\cap K \\
\frac{\partial \chi_i}{\partial n}&=0 &&\quad\mbox{ on }\partial Q_1^{\epsilon}\cap K\nonumber\\
\chi_i &= g_i &&\quad\text{ on }\partial K\backslash \partial Q_1^{\epsilon}. \nonumber
\end{alignat}
Here $g_i$ is affine over $\partial K$ with $g_i(O_j)=\delta_{ij}$ for
all $i,j=1,\cdots, N$ and we recall that $\{O_j\}_{j=1}^{N}$ are the set of coarse nodes on $\mathcal{T}_{H}$.

Algorithm \ref{algorithm:wavelet} proceeds as follows: We first
construct the local multiscale basis functions on each coarse
neighborhood $\omega_i$. Given level parameter $\ell\in \mathbb{N}$,
and the four coarse edges $\Gamma_{i,k}$ with $k=1,2,3,4$, i.e.,
$\cup_{k=1}^{4}\Gamma_{i,k}=\partial\omega_i$, we let $V_{\ell,k}^{i}$ be the
Haar wavelet up to level $\ell$ on the coarse edge
$\Gamma_{i,k}$. Introducing $V_{i,\ell}:=\oplus_{k=1}^{4}V_{\ell,k}^{i}$
to be the edge basis functions on $\partial\omega_i$, then $V_{i,\ell}$ becomes a good approximation space of dimension $2^{\ell+2}$ to the trace of the solution over $\partial \omega_i$, i.e., $u|_{\partial\omega_i}$.

Subsequently, we calculate the local multicale basis functions on each coarse neighborhood $\omega_i$ with all possible Dirichlet boundary conditions in $V_{i,\ell}$, and denote the resulting local multiscale space as $\mathcal{L}^{-1}_i (V_{i,\ell})$ in Step 2. We can then define the global multiscale space as $V_{\text{ms},\ell}^{\text{EW}}$ and obtain the multiscale solution $u_{\text{ms},\ell}^{\text{EW}}$ in Steps 3 and 4.

\begin{algorithm}[H]
\caption{Wavelet-based Edge Multiscale Finite Element Method (WEMsFEM)}
\label{algorithm:wavelet}
 \begin{tabular}{l l}
\specialrule{.2em}{.1em}{.1em}
\textbf{Input}:&The level parameter $\ell\in \mathbb{N}$; coarse neighborhood $\omega_i$ and its four coarse edges $\Gamma_{i,k}$ with \\&
    $k=1,2,3,4$, i.e., $\cup_{k=1}^{4}\Gamma_{i,k}=\partial\omega_i$;
    the subspace $V_{\ell,k}^i\subset L^2(\Gamma_{i,k})$ up to level $\ell$ on each\\& coarse edge $\Gamma_{i,k}$;
    \\
    \textbf{Output}:& Multiscale solution $u_{\text{ms},\ell}^{\text{EW}}$.\\
\specialrule{.2em}{.1em}{.1em}
1. & Denote
    $V_{i,\ell}:=\oplus_{k=1}^{4}V_{\ell,k}.$
    Then the number of basis functions in $V_{i,\ell}$ is $4\times 2^{\ell}=2^{\ell+2}$. \\&
    Denote these basis
    functions as $v_k$ for $k=1,\cdots, 2^{\ell+2}$.\\
2. &Calculate local multiscale basis $\mathcal{L}^{-1}_i (v_k)$ for all $k=1,\cdots,2^{\ell+2}$.\\
&$\mathcal{L}^{-1}_i (v_k):=v$ satisfies:\\
& $
  \left\{ \begin{aligned}
          \mathcal{L}_i v&:=\Delta v+k^2 v=0&&\mbox{in }\Omega^{\epsilon}\cap\omega_i\\
          \frac{\partial v}{\partial n}&=0 &&\mbox{ on }\partial Q_1^{\epsilon}\cap \omega_i\\
          v&=v_k &&\mbox{on }\partial\omega_i\backslash \partial Q_1^{\epsilon}.
  \end{aligned}\right.
$\\
&Calculate one local solution $v^i$ defined by the solution to the local problem:\\
&$
\left\{
\begin{aligned}
 -(\Delta+k^2)v^{i}&=1 \quad&&\text{ in } \Omega^{\epsilon}\cap\omega_i\\
\frac{\partial v^{i}}{\partial n}&=0 &&\mbox{ on }\partial Q_1^{\epsilon}\cap \omega_i\\
v^{i}&=0\quad&&\text{ on }\partial\omega_i\backslash \partial Q_1^{\epsilon}.
\end{aligned}
\right.
$\\
3. &Build the ansatz space. \\&
$
V_{\text{ms},\ell}^{\rm EW}  := \text{span} \{\chi_i\mathcal{L}^{-1}_i(v_k), \chi_i v^{i}: \,  \, 1 \leq i \leq N,\,\,\, 1 \leq k \leq  2^{\ell+2}\}.
$\\
4. &Solve for \eqref{eq:globalG} by the Conforming Galerkin method in
$V_{\text{ms},\ell}^{\rm EW}$ to obtain $u_{\text{ms},\ell}^{\text{EW}}$.\\[0cm]
\specialrule{.2em}{.1em}{.1em}
\end{tabular}
\end{algorithm}
Note that all the global multiscale basis functions in $V_{\text{ms},\ell}^{\rm EW}$ fulfill the interface condition, i.e. they satisfy the Neumman boundary condition on every interface of every inclusion, as in the second equation in \eqref{eq:model}. This is due to the construction of the partition of unity functions $\chi_i$ in \eqref{pou}.

\subsection{Local projection}
An important element of our algorithm is the knowledge of the approximation properties to the edge basis functions $V_{i,\ell}:=\oplus_{k=1}^{4}V_{\ell,k}$ on each coarse neighborhood $\partial\omega_i$.

Let the $L^2(\partial\omega_i)$-orthogonal projection $\mathcal{P}_{i,\ell}$ onto the local multiscale space up to level $\ell$: $L^2(\partial\omega_i)\to V_{i,\ell}$ be defined by
\begin{align}\label{eq:projectionEDGE}
\mathcal{P}_{i,\ell}(v):=\sum\limits_{j=1}^{2^{\ell+2}}(v,\psi_{j})_{\partial\omega_i}\mathcal{L}_{i}^{-1}(\psi_{j}) \quad \text{ for all }v\in L^2(\partial\omega_i).
\end{align}
Here, we denote $\psi_{j}$ for $j=1,\cdots 2^{\ell+2}$ as the Haar wavelet defined on the four edges of $\omega_i$ of level $\ell$ and the local operator $\mathcal{L}_i$ is defined in Algorithm \ref{algorithm:wavelet}.

Let $\text{diam}(\Omega^{\epsilon})$ be the diameter of the bounded domain $\Omega^{\epsilon}$. Define
\begin{align*}
\Cpoin{\omega_i}&:=H^{-2}\max\limits_{w\in H^1_0(\Omega^{\epsilon}\cap\omega_i)}\frac{\int_{\Omega^{\epsilon}\cap\omega_i}w^2\dx}
{\int_{\Omega^{\epsilon}\cap\omega_i}|\nabla w|^2\dx},
\\
\Cpoin{\Omega^{\epsilon}}&:=\text{diam}(\Omega^{\epsilon})^{-2}\max\limits_{w\in H^1_0(\Omega^{\epsilon})}\frac{\int_{\Omega^{\epsilon}}w^2\dx}{\int_{\Omega^{\epsilon}}|\nabla w|^2\dx}.
\end{align*}
Then the positive constants $\Cpoin{\omega_i}$ and
$\Cpoin{\Omega^{\epsilon}}$ are independent of the wavenumber $k$ and
the coarse mesh $\mathcal{T}_{H}$. Note that we will utilize the same
constant $\Cpoin{\omega_i}$ to denote the constant from the Poincar\'{e} inequality.

\begin{assumption}[Scale Resolution Assumption]\label{ass:resolution}
We assume that the coarse mesh size $H$ is sufficiently small to satisfy the following inequality
\begin{align*}
\max_{i=1,2,\ldots,N}\{\Cpoinn{\omega_i}{1/2}\}Hk<1.
\end{align*}
\end{assumption}
To simplify the notation, we denote
\[
\Ce:=\Big(1-\max_{i=1,2,\ldots,N}\{\Cpoin{\omega_i}\}(Hk)^2\Big)^{-1}.
\]
\begin{remark}
Similar resolution assumption as Assumption \ref{ass:resolution} is commonly seen in the literature, e.g.,  \cite{MR2812565}. If we take the coarse scale mesh grid $H\ll k^{-1}$, then $\Ce\approx 1$.
\end{remark}

\section{Convergence rate of Algorithm \ref{algorithm:wavelet}}\label{sec:convergence}
The convergence rate of this algorithm is clearly an important detail
and, perhaps remarkably, this can be obtained. The proof is inspired
by the techniques developed in \cite{fu2018edge,GL18}, where a local
decomposition of the solution $u$ restricted on each coarse
neighborhood $\omega_i$ for $i=1,2,\ldots,N$, namely, $u|_{\omega_i}$,
is introduced. We also analyze the local approximation properties of the multiscale basis functions constructed in Step 2, Algorithm \ref{algorithm:wavelet} to each component of the decomposition of $u|_{\omega_i}$ in Section \ref{subsec:local-decmp}. Subsequently, the global approximation properties of the ansatz space $V_{\text{ms},\ell}^{\rm EW}$ and the convergence of Algorithm \ref{algorithm:wavelet} are investigated in Section \ref{subsec:ansatz}.


\subsection{Local decomposition of the solution}\label{subsec:local-decmp}

The solution $u$ satisfies the following equation
\begin{equation*}
\begin{aligned}
-(\Delta +k^2)u&=f \quad&&\text{ in } \Omega^{\epsilon}\cap\omega_i,
\end{aligned}
\end{equation*}
which can be split into three parts, namely
\begin{align}\label{eq:decomp1}
u|_{\omega_i}=u^{i,\roma}+u^{i,\romb}+u^{i,\romc}.
\end{align}
Here, the three components $u^{i,\roma}$, $u^{i,\romb}$ and $u^{i,\romc}$ are respectively given by
\begin{equation}\label{eq:u-roma1conv}
\left\{
\begin{aligned}
-(\Delta +k^2) u^{i,\roma}&=f-\dashint_{\omega_i}f \quad&&\text{ in } \Omega^{\epsilon}\cap\omega_i\\
\frac{\partial u^{i,\roma}}{\partial n}&=0 &&\mbox{ on }\partial Q_1^{\epsilon}\cap \omega_i\\
u^{i,\roma}&=0\quad&&\text{ on }\partial \omega_i\backslash \partial Q_1^{\epsilon},
\end{aligned}
\right.
\end{equation}
\begin{equation*}
\left\{
\begin{aligned}
-(\Delta +k^2) u^{i,\romb}&=0 \quad&&\text{ in } \Omega^{\epsilon}\cap\omega_i\\
\frac{\partial u^{i,\romb}}{\partial n}&=0 &&\mbox{ on }\partial Q_1^{\epsilon}\cap \omega_i\\
u^{i,\romb}&=u\quad&&\text{ on }\partial \omega_i\backslash \partial Q_1^{\epsilon},
\end{aligned}
\right.
\end{equation*}
and
\[
u^{i,\romc}=v^{i}\dashint_{\omega_i}f.
\]
Here, $\dashint_{\omega_i}v:=|\omega_i|^{-1}\int_{\omega_i}v\dx$ denotes the
average of the function $v\in L^1(\omega_i)$ over each coarse neighborhood $\omega_i$. Recall that $v^{i}$ is defined in Step 2 of Algorithm \ref{algorithm:wavelet}.

We first show that $u^{i,\roma}$ is negligible thanks to the local basis function $v^{i}$:
\begin{lemma}\label{lem:parta}
Let the Scale Resolution Assumption \ref{ass:resolution} be valid. Let $u^{i,\roma}$ be defined in \eqref{eq:u-roma1conv} and let $f\in L^2(\Omega^{\epsilon})$. Then there holds
\begin{align*}
\|{u^{i,\roma}}\|_{V(\Omega^{\epsilon}\cap\omega_i)}\leq 2\Ce\Cpoinn{\omega_i}{1/2}H\normL{f}{\Omega^{\epsilon}\cap\omega_i}.
\end{align*}
\end{lemma}
\begin{proof}
Multiplying \eqref{eq:u-roma1conv} by $u^{i,\roma}$, and taking its integral over $\omega_i$, leads to
\begin{align*}
\normHsemi{u^{i,\roma}}{\Omega^{\epsilon}\cap\omega_i}^2
&=k^2\normL{u^{i,\roma}}{\Omega^{\epsilon}\cap\omega_i}^2+\int_{\Omega^{\epsilon}\cap\omega_i}(f-\dashint_{\Omega^{\epsilon}\cap\omega_i}f) u^{i,\roma}\dx\\
&=k^2\normL{u^{i,\roma}}{\Omega^{\epsilon}\cap\omega_i}^2+\int_{\Omega^{\epsilon}\cap\omega_i}f(u^{i,\roma}-\dashint_{\Omega^{\epsilon}\cap\omega_i}u)\dx
\end{align*}
wherein an application of  H\"{o}lder's inequality and the Poincar\'{e} inequality proves
\begin{align*}
\normHsemi{u^{i,\roma}}{\Omega^{\epsilon}\cap\omega_i}^2
\leq (\Cpoinn{\omega_i}{1/2}Hk)^2\normHsemi{u^{i,\roma}}{\Omega^{\epsilon}\cap\omega_i}^2+\Cpoinn{\omega_i}{1/2}H\normL{f}{\Omega^{\epsilon}\cap\omega_i}
\normHsemi{u^{i,\roma}}{\Omega^{\epsilon}\cap\omega_i}.
\end{align*}
After moving the first term on the right of the previous estimate to the left, and noting the Scale Resolution Assumption \ref{ass:resolution}, we obtain
\begin{align*}
\normHsemi{u^{i,\roma}}{\Omega^{\epsilon}\cap\omega_i}\leq \Ce\Cpoinn{\omega_i}{1/2}H\normL{f}{\Omega^{\epsilon}\cap\omega_i}.
\end{align*}
Finally, an application of the Poincar\'{e} inequality and the Scale Resolution Assumption \ref{ass:resolution} shows the desired assertion.
\end{proof}
Recalling that $z\in V$ is the solution to Problem \eqref{eqn:weakformDual} for given $w\in L^2(\Omega^{\epsilon};\mathbb{C})$, analogously to Decomposition \eqref{eq:decomp1}, the following decomposition is valid:
\begin{align}\label{eq:decompdual}
z|_{\omega_i}=z^{i,\roma}+z^{i,\romb}+z^{i,\romc}.
\end{align}
Then a similar argument as in Lemma \ref{lem:parta} leads to the following estimate:
\begin{lemma}\label{lem:parta-dual}
Let the Scale Resolution Assumption \ref{ass:resolution} hold. Let
$z^{i,\roma}$ be defined in \eqref{eq:decompdual} and let $w\in
L^2(\Omega^{\epsilon})$. Then it holds that
\begin{align*}
\|{z^{i,\roma}}\|_{V(\Omega^{\epsilon}\cap\omega_i)}\leq 2\Ce\Cpoinn{\omega_i}{1/2}H\normL{w}{\Omega^{\epsilon}\cap\omega_i}.
\end{align*}
\end{lemma}
\begin{remark}[On the decomposition \eqref{eq:decomp1} and \eqref{eq:decompdual}]
Following from Lemmas \ref{lem:parta} and \ref{lem:parta-dual}, only
one multiscale interial basis function in Step 2 of Algorithm
\ref{algorithm:wavelet} is sufficient if $\mathcal{O}(H)$ convergence
rate is desired or sufficient. Otherwise, one can construct extra local multiscale basis functions to approximate the first component $u^{i,\roma}$.

\end{remark}
Since $u^{i,\romc}$ is of rank-one, which can be represented with one multiscale basis function $v^i$ in Decomposition
\eqref{eq:decomp1}. Lemma \ref{lem:parta} indicates that we need only
  construct a proper ansatz space for the second part $u^{i,\romb}$
  to ensure a good ansatz space for $u|_{\omega_i}$. We now prove that the multiscale basis functions, constructed in Step 2 of Algorithm \ref{algorithm:wavelet}, span an appropriate ansatz space with good approximation properties:
\begin{theorem}[Approximation properties of the projection $\mathcal{P}_{i,\ell}$] \label{thm:proj}
Let Assumption \ref{ass:resolution} hold and let $e\in V(\Omega^{\epsilon}\cap\omega_i)$ satisfy
\begin{equation}\label{eq:pde-approx}
\left\{
\begin{aligned}
\mathcal{L}_i e&:=\Delta e+k^2 e=0&&\mbox{in }\Omega^{\epsilon}\cap\omega_i\\
\frac{\partial e}{\partial n}&=0 &&\mbox{ on }\partial Q_1^{\epsilon}\cap \omega_i\\
 e&=u^{i,\romb}-\mathcal{P}_{i,\ell}(u^{i,\romb})&&\mbox{on }\partial\omega_i\backslash \partial Q_1^{\epsilon}.
\end{aligned}
\right.
\end{equation}
Then it holds that
\begin{align*}
\normL{e}{\Omega^{\epsilon}\cap\omega_i}&\leq\Cw \Ce2^{-\ell/2}H\Big(\|u\|_{H^1( \Omega^{\epsilon}\cap \omega_i)}+\Ce\Cpoinn{\omega_i}{1/2}\normL{f}{\Omega^{\epsilon}\cap\omega_i}\Big)\\
\normL{\nabla(\chi_i\locv{e}{}{})}{\Omega^{\epsilon}\cap\omega_i}
&\leq\Cw\Ce 2^{-\ell/2}\Big(\|u\|_{H^1( \Omega^{\epsilon}\cap\omega_i)}+\Ce\Cpoinn{\omega_i}{1/2}\normL{f}{\Omega^{\epsilon}\cap\omega_i}\Big).
\end{align*}
\end{theorem}
\begin{proof}
We can obtain from Theorem \ref{lem:very-weak}:
\begin{align*}
\normL{e}{\Omega^{\epsilon}\cap\omega_i}&\leq\Cw \Ce H^{1/2}\|u^{i,\romb}-\mathcal{P}_{i,\ell}(u^{i,\romb})\|_{\partial\omega_i\backslash \partial Q_1^{\epsilon}}\\
\normL{\nabla(\chi_i\locv{e}{}{})}{\Omega^{\epsilon}\cap\omega_i}
&\leq\Cw\Ce H^{-1/2}\|u^{i,\romb}-\mathcal{P}_{i,\ell}(u^{i,\romb})\|_{\partial\omega_i\backslash \partial Q_1^{\epsilon}}.
\end{align*}
Whereas an application of \cite[Eqn. (5.6)]{fu2018edge} leads to
\begin{equation}\label{eq:11111}
\begin{aligned}
\normL{e}{\Omega^{\epsilon}\cap\omega_i}&\leq\Cw \Ce 2^{-\ell/2}H\normHsemi{u^{i,\romb}}{\Omega^{\epsilon}\cap\omega_i}\\
\normL{\nabla(\chi_i\locv{e}{}{})}{\Omega^{\epsilon}\cap\omega_i}
&\leq\Cw\Ce 2^{-\ell/2}\normHsemi{u^{i,\romb}}{\Omega^{\epsilon}\cap\omega_i}.
\end{aligned}
\end{equation}
We only need to estimate $\normHsemi{u^{i,\romb}}{\omega_i}$. Note that $r:=u-u^{i,\romb}$ satisfies
\begin{equation*}
\left\{
\begin{aligned}
\Delta r+k^2 r&=f\quad&&\text{ in } \Omega^{\epsilon}\cap\omega_i\\
\frac{\partial r}{\partial n}&=0 &&\mbox{ on }\partial Q_1^{\epsilon}\cap \omega_i\\
r&=0\quad&&\text{ on }\partial \omega_i\backslash\partial Q_1^{\epsilon}.
\end{aligned}
\right.
\end{equation*}
A similar estimate as used in the proof to Lemma \ref{lem:parta} shows
that
\[
\normHsemi{r}{\Omega^{\epsilon}\cap\omega_i}\leq \Ce \Cpoinn{\omega_i}{1/2}\normL{f}{\Omega^{\epsilon}\cap\omega_i}.
\]
Finally, an application of the triangle inequality proves
\begin{align*}
\normHsemi{u^{i,\romb}}{\Omega^{\epsilon}\cap\omega_i}&\leq \normHsemi{u}{\Omega^{\epsilon}\cap\omega_i}+\normHsemi{r}{\Omega^{\epsilon}\cap\omega_i}\\
&\leq \normHsemi{u}{\Omega^{\epsilon}\cap\omega_i}+\Ce \Cpoinn{\omega_i}{1/2}\normL{f}{\Omega^{\epsilon}\cap\omega_i}
\end{align*}
and this, together with \eqref{eq:11111}, proves the desired assertion.
\end{proof}

\subsection{Approximation properties of the ansatz space $V_{\text{ms},\ell}^{\rm EW}$}\label{subsec:ansatz}
The approximation properties of the ansatz space
$V_{\text{ms},\ell}^{\rm EW}$ follow from the theorem below:
\begin{theorem}[Approximation properties of the multiscale space $V_{\text{ms},\ell}^{\rm EW}$] \label{thm:last}
Let Assumption \ref{ass:resolution} hold and
assume that $f\in L^2(\Omega^{\epsilon})$. Let $\ell\in
\mathbb{N}_{+}$ and $u\in V$ be the solution to Problem
\eqref{eq:model}. There then holds
\begin{equation}\label{eq:waveletErrconv1}
\begin{aligned}
\inf\limits_{v\in V_{\rm{ms},\ell}^{\rm EW}}\|u-v\|_V
&\leq \Const{\ref{thm:last}} \Ce \Big(H+\Ce 2^{-\ell/2}\Co\Big)\|{f}\|_{L^2(\Omega^{\epsilon})}.
\end{aligned}
\end{equation}
Furthermore, let $w\in L^2(\Omega^{\epsilon})$ and let $z\in V$ be the solution to Problem \eqref{eqn:weakformDual}. Then it holds
\begin{equation}\label{eq:waveletErrconvDual}
\inf\limits_{v\in V_{\rm{ms},\ell}^{\rm EW}}\|z-v\|_V\leq
\Const{\ref{thm:last}} \Ce \Big(H+\Ce 2^{-\ell/2}\Co\Big)\|{w}\|_{L^2(\Omega^{\epsilon})}.
\end{equation}
{Here, $\Const{\ref{thm:last}}$ is a positive constant independent of the wavenumber $k$ or coarse mesh size $H$. }
\end{theorem}
\begin{proof}
We will only prove the first assertion \eqref{eq:waveletErrconv1} since the second assertion \eqref{eq:waveletErrconvDual} can be derived in a similar manner.

Recall the local decomposition in \eqref{eq:decomp1} on each coarse neighborhood $\omega_i$ for $i=1,\ldots,N$. Let
\[v:=\sum\limits_{i=1}^{N}\chi_i\mathcal{P}_{i,\ell}u^{i,\romb}+\chi_iv^{i}\dashint_{\omega_i}f
\in V_{\text{ms},\ell}^{\rm EW}.\]
We prove that $v$ is a good approximation to $u$.

Denote $\locv{e}{}{}:=u-v$.
Then the property of the partition of unity of $\{\chi_i\}_{i=1}^{N}$ leads to
\[
\locv{e}{}{}=\sum\limits_{i=1}^{N}\chi_i\locv{e}{i}{} \qquad\text{ with }
\qquad\locv{e}{i}{}:=u^{i,\roma}+(u^{i,\romb}-\mc{P}_{i,\ell}u^{i,\romb})
:=\locv{e}{i}{\roma}+\locv{e}{i}{\romb}.
\]
Taking its squared energy norm, and using the overlap condition \eqref{eq:overlap}, we arrive at
\begin{align*}
\|e\|_V^2&=\int_{\Omega^{\epsilon}}\Big(|\nabla \locv{e}{}{}|^2+k^2 e^2 \Big) \dx=\int_{\Omega^{\epsilon}}
\Big(
\Big|\sum\limits_{i=1}^{N}
\nabla(\chi_i\locv{e}{i}{})\Big|^2+k^2 \Big(\sum\limits_{i=1}^{N}\chi_i\locv{e}{i}{}\Big)^2
\Big)
\dx\\
&\leq\Cov\sum\limits_{i=1}^{N}\Big(\int_{\Omega^{\epsilon}\cap\omega_i}
|\nabla(\chi_i\locv{e}{i}{})|^2\dx+k^2\int_{\Omega^{\epsilon}\cap\omega_i}
(\locv{e}{i}{})^2\dx
\Big).
\end{align*}
Then Young's inequality implies
\begin{align}\label{eq:1111}
\int_{\Omega^{\epsilon}\cap\omega_i}
|\nabla(\chi_i\locv{e}{i}{})|^2\dx+k^2\int_{\Omega^{\epsilon}\cap\omega_i}
(\locv{e}{i}{})^2\dx
&\leq 2\Big(\int_{\Omega^{\epsilon}\cap\omega_i}
|\nabla(\chi_i\locv{e}{i}{\roma})|^2\dx+ k^2\int_{\Omega^{\epsilon}\cap\omega_i}
(\locv{e}{i}{\roma})^2\dx\Big)\nonumber\\
&+2\Big( \int_{\Omega^{\epsilon}\cap\omega_i}
|\nabla(\chi_i\locv{e}{i}{\romb})|^2\dx+k^2\int_{\Omega^{\epsilon}\cap\omega_i}
(\locv{e}{i}{\romb})^2\dx
 \Big).
\end{align}
Using the product rule, and the Poincar\'{e} inequality, we obtain
\begin{align}
\int_{\Omega^{\epsilon}\cap\omega_i}
|\nabla(\chi_i\locv{e}{i}{\roma})|^2\dx+ k^2\int_{\Omega^{\epsilon}\cap\omega_i}
(\locv{e}{i}{\roma})^2\dx
&\leq 2\Big( \int_{\Omega^{\epsilon}\cap\omega_i}
|\nabla\chi_i|^2|\locv{e}{i}{\roma}|^2\dx+\|\locv{e}{i}{\roma}\|_{V(\Omega^{\epsilon}\cap\omega_i)}^2\Big)\nonumber\\
&\leq 2\Big(\Cpoin{\omega_i}\int_{\Omega^{\epsilon}\cap\omega_i}|\nabla\locv{e}{i}{\roma}|^2\dx
+\|\locv{e}{i}{\roma}\|_{V(\Omega^{\epsilon}\cap\omega_i)}^2\Big).\label{eq:444}
\end{align}
Then Lemma \ref{lem:parta} yields
\begin{align*}
\int_{\Omega^{\epsilon}\cap\omega_i}
|\nabla(\chi_i\locv{e}{i}{\roma})|^2\dx+ k^2\int_{\Omega^{\epsilon}\cap\omega_i}
(\locv{e}{i}{\roma})^2\dx
\leq (8+2\Cpoin{\omega_i})\Ce^2\Cpoin{\omega_i}H^2\normL{f}{\Omega^{\epsilon}\cap\omega_i}^2.
\end{align*}
Analogously, we can derive the following upper bound for the second term by Theorem \ref{thm:proj}:
\begin{align*}
\normL{\nabla(\chi_i\locv{e}{i}{\romb})}{\Omega^{\epsilon}\cap\omega_i}^2+ k^2\int_{\Omega^{\epsilon}\cap\omega_i}
(\locv{e}{i}{\romb})^2\dx
\leq\Cw\Ce^22^{-\ell}\Big(\|u\|_{H^1(\Omega^{\epsilon}\cap \omega_i)}^2+\Ce^2\Cpoin{\omega_i}\normL{f}{\Omega^{\epsilon}\cap\omega_i}^2\Big).
\end{align*}
Inserting these two estimates into \eqref{eq:1111}, and utilizing the overlapping condition \eqref{eq:overlap}, leads to
\begin{align}\label{eq:999}
\|e\|_V^2
&\leq\Cw\Ce^2\Big((H^2+\Ce^22^{-\ell})\normL{f}{\Omega^{\epsilon}}^2
+2^{-\ell}\normHsemi{u}{\Omega^{\epsilon}}^2\Big).
\end{align}
Furthermore, inserting \eqref{res:apriori} into \eqref{eq:999} shows \eqref{eq:waveletErrconv1}. This completes the proof.
\end{proof}

Finally, we are ready to present the main result of this section:
\begin{proposition}[Error estimate for Algorithm \ref{algorithm:wavelet}]\label{prop:wavelet-basedconv}
Assume that $f\in L^2(\Omega^{\epsilon})$ and let the coarse mesh size $H$ and the level parameter $\ell\in \mathbb{N}_{+}$ satisfy
\begin{align}\label{eq:wavelet-level}
H\leq k^{-1}\frac{1}{4\Const{b}\Const{\ref{thm:last}}\Ce}\quad \text{ and }\quad\ell\geq 2\log_2(\Const{b}\Const{\ref{thm:last}}\Ce^2k\Co)+4.
\end{align}
Let $u\in V$ and $u_H\in V_{\rm{ms},\ell}^{\rm EW}$ be the solution to Problem \eqref{eq:model}, and from Algorithm \ref{algorithm:wavelet}, respectively. There holds
\begin{equation}\label{eq:waveletErrconv}
\begin{aligned}
\|u-u_{\text{ms},\ell}^{\text{EW}}\|_V
&\leq 4 \Const{b} \Const{\ref{thm:last}} \Ce H\|{f}\|_{L^2(\Omega^{\epsilon})}.
\end{aligned}
\end{equation}
\end{proposition}
\begin{proof}
Since $V_{\text{ms},\ell}^{\rm EW}\subset V$, G\aa rding\rq{}s inequality in Theorem \ref{them:sesquilinear} implies
\begin{align}
\|u-u_{\text{ms},\ell}^{\text{EW}}\|_V^2&\leq a(u-u_{\text{ms},\ell}^{\text{EW}},u-u_{\text{ms},\ell}^{\text{EW}})
+2k^2\normL{u-u_{\text{ms},\ell}^{\text{EW}}}{\Omega^{\epsilon}}^2\nonumber\\
&=a(u-u_{\text{ms},\ell}^{\text{EW}},u-v_{\text{ms},\ell}^{\text{EW}})
+2k^2\normL{u-u_{\text{ms},\ell}^{\text{EW}}}{\Omega^{\epsilon}}^2\label{eq:222}
\end{align}
for all $v_{\text{ms},\ell}^{\text{EW}}\in V_{\text{ms},\ell}^{\text{EW}}$.

Next we estimate $\normL{u-u_{\text{ms},\ell}^{\text{EW}}}{\Omega^{\epsilon}}$: Let $z\in V$ be the solution to
\begin{align*}
a(v,z)=(v,u-u_{\text{ms},\ell}^{\text{EW}})_{\Omega^{\epsilon}} \text{ for all } v\in V,
\end{align*}
then for all $z_{\text{ms},\ell}^{\text{EW}}\in
V_{\text{ms},\ell}^{\text{EW}}$ we obtain that
\begin{align*}
\normL{u-u_{\text{ms},\ell}^{\text{EW}}}{\Omega^{\epsilon}}^2
&=a(u-u_{\text{ms},\ell}^{\text{EW}},z)
=a(u-u_{\text{ms},\ell}^{\text{EW}},z-z_{\text{ms},\ell}^{\text{EW}}).
\end{align*}
Furthermore, an application of Theorem \ref{them:sesquilinear} leads to
\begin{align*}
\normL{u-u_{\text{ms},\ell}^{\text{EW}}}{\Omega^{\epsilon}}^2
&\leq \Const{b}\|u-u_{\text{ms},\ell}^{\text{EW}}\|_V\inf\limits_{z_{\text{ms},\ell}^{\text{EW}}\in V_{\text{ms},\ell}^{\text{EW}}}\|z-z_{\text{ms},\ell}^{\text{EW}}\|_V.
\end{align*}
By \eqref{eq:waveletErrconvDual}, and condition \eqref{eq:wavelet-level}, we arrive at
\begin{align*}
\normL{u-u_{\text{ms},\ell}^{\text{EW}}}{\Omega^{\epsilon}}
&\leq \frac{1}{2k}\|u-u_{\text{ms},\ell}^{\text{EW}}\|_V.
\end{align*}
This, together with \eqref{eq:222}, leads to
\begin{align*}
\|u-u_{\text{ms},\ell}^{\text{EW}}\|_V^2\leq a(u-u_{\text{ms},\ell}^{\text{EW}},u-v_{\text{ms},\ell}^{\text{EW}})
+\frac{1}{2}\|u-u_{\text{ms},\ell}^{\text{EW}}\|_V^2.
\end{align*}
Consequently, we obtain
\begin{align*}
\|u-u_{\text{ms},\ell}^{\text{EW}}\|_V^2\leq 2 a(u-u_{\text{ms},\ell}^{\text{EW}},u-v_{\text{ms},\ell}^{\text{EW}}).
\end{align*}
Finally, an application of the boundedness of the sesquilinear form $a(\cdot,\cdot)$ Theorem \ref{them:sesquilinear} and the approximation property \eqref{eq:waveletErrconv} prove the desired assertion.
\end{proof}
\section{Numerical experiments}\label{sec:numerical}

We present numerical experiments to show the
performance of Algorithm \ref{algorithm:wavelet} with locally periodic perforations in Sections \ref{subsec:levelpara} and \ref{subsec:meshsize}, and with random perforations in Section \ref{subsec:random}. We consider two different source excitations placed at $(0,0)$ and $(1.44,0)$ for each model, and take the wavenumber
$k=64$.

The coarse mesh $\mathcal{T}_H$ is a regular partition of the domain $\Omega^{\epsilon}$ into
finite elements with a mesh size $H$. Then each coarse element is further partitioned
into a union of connected fine grid blocks. The fine-grid partition is denoted by
$\mathcal{T}_h$, which provides a sufficiently fine mesh for standard
finite element solvers to get a reference solution; for sufficient
accuracy we take $h:={1}/{320}$. In addition, the Perfectly Matched Layer (PML) is utilized to absorb the outgoing wave, see Appendix \ref{sec:pml} for more details.

As is usual in the finite element literature we use the $L^2(\Omega^{\epsilon})$ and
$H^1(\Omega^{\epsilon})$-relative errors to assess the accuracy of the scheme, which are defined by
\begin{align*}
\frac{\normL{u_h-u_{\text{ms},\ell}^{\text{EW}}}{\Omega^{\epsilon}}}{\normL{u_h}{\Omega^{\epsilon}}}
\quad \text{ and }\quad \frac{\normL{\nabla (u_h-u_{\text{ms},\ell}^{\text{EW}})}{\Omega^{\epsilon}}}{\normL{\nabla u_h}{\Omega^{\epsilon}}}.
\end{align*}

\subsection{Performance of Algorithm \ref{algorithm:wavelet}: level
  parameter $\ell$}\label{subsec:levelpara}
We consider in this section the two perforated models as shown
in Figure \ref{fig:testmodel}. To describe the computational domain
$\Omega^{\epsilon}$, we use Eq.\eqref{eq:geo}. Let
$\Omega:=[-2.4,2.4]^2$, $\Omega_1:=[-1,1]^2$ and the size of the cell
$\epsilon:=1/6$. The perforations in a unit cell are
$Q:=[0.25,0.75]^2$ and $Q:=[0.1,0.9]^2$ in models 1 and 2. Note that
there are $144$ perforations in both models, this strong heterogeneity
in the computational domain $\Omega^{\epsilon}$ makes Problem
\eqref{eq:model} even harder.  

Figure \ref{fig:solcompare_m1_c} demonstrates
 dynamic anisotropy, also known as self-collimation, \cite{chigrin03a,prather07a} which is a striking, frequency sensitive and dependent,
 effect. Naively, one might assume that wave
 excitation of the crystal, at its centre, would lead to isotropic
 wavefronts within the crystal; this is indeed the case in the standard
 homogenisation, low-frequency long-wave, limit where the wavelength
 is much larger than the cell-to-cell spacing.  However, as the
 frequency increases and wavelength decreases,
 Bragg scattering occurs with constructive or destructive interference
 leading to well-defined frequency windows (band-gaps) within which
 wave propagation is disallowed. Interference also occurs that acts to
 create anisotropic wavefronts with the most severe example being that
 where all the wave energy is channeled in specific directions; in
 terms of homogenisation there are variants that are developed for
 high-frequencies \cite{craster10a} that show the effective medium PDE changing its
 character from elliptic to hyperbolic with these directions of
 self-collimation being the characteristics of the hyperbolic
 system \cite{maling17a}. Figure \ref{fig:solcompare_m1_r} shows
 lensing, another effect created by the crystal whereby a partial image of
 the source (on the right) forms to the left of the crystal.

\begin{figure}[H]
	\centering
	\subfigure[Reference solution.]{
		\includegraphics[trim={4cm 9.5cm 4cm 10cm},clip,width=2.5in]{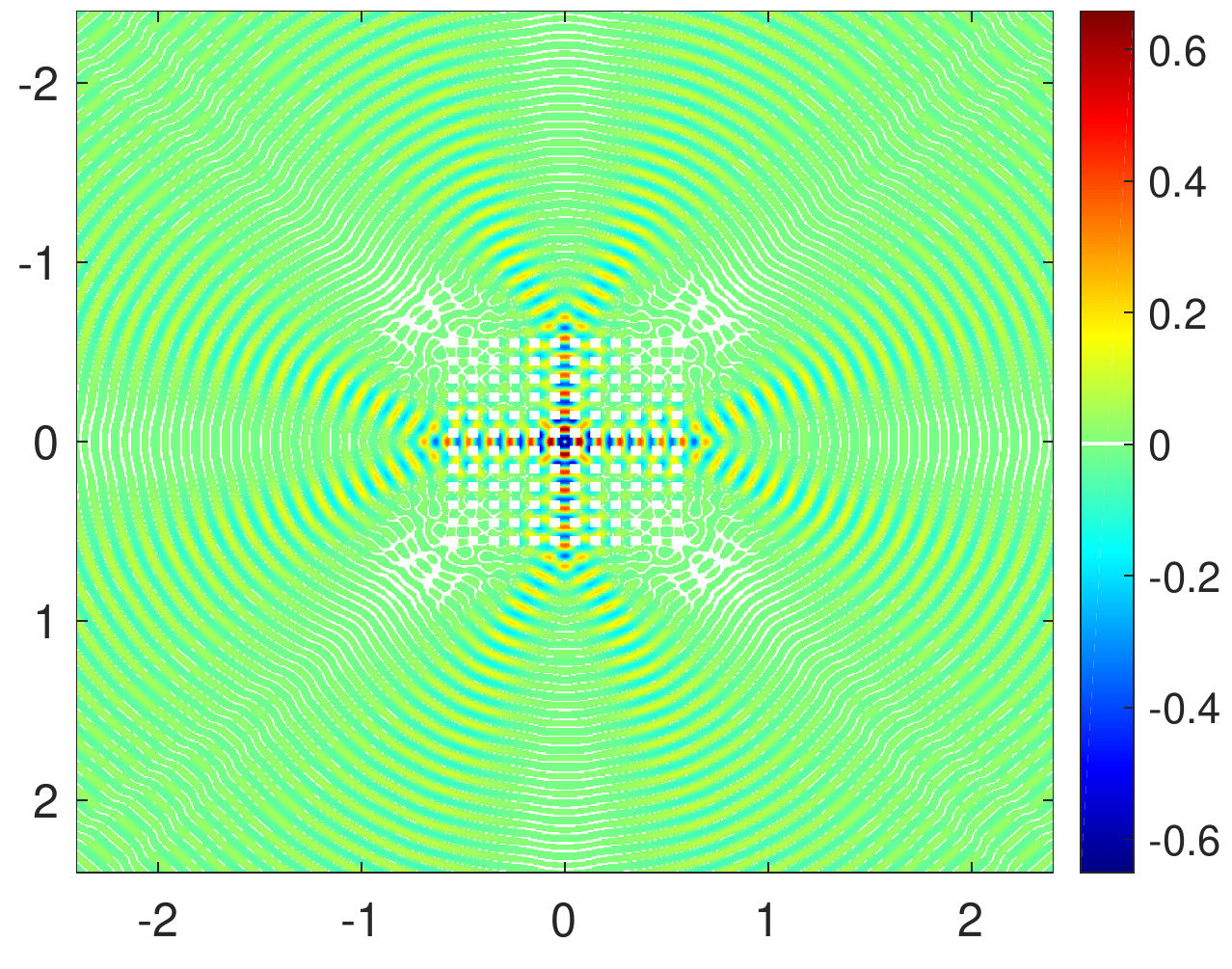}}
	\subfigure[Multiscale solution.]{
		\includegraphics[trim={4cm 9.5cm 4cm 10cm},clip,width=2.5in]{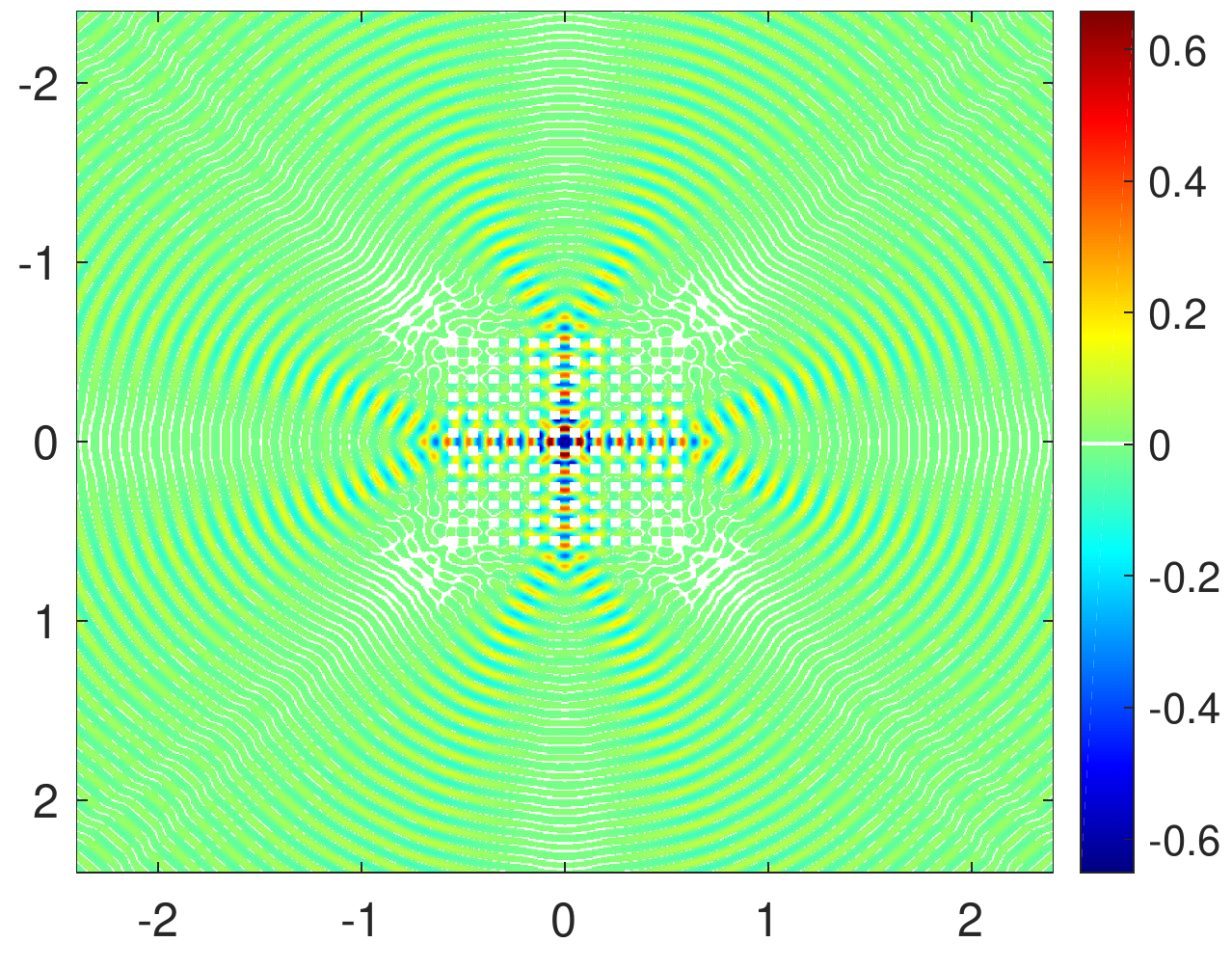}}	
	\caption{Reference solution and multiscale solution for model
          1 with centered source, $H:=1/10$ and $\ell=2$. The
          $L^2(\Omega^{\epsilon})$-relative error is 5.35\%. This
          simulation demonstrates the highly directional wavefields
          created by dynamic anisotropy. The photonic crystal is
          outlined as the small white squares, c.f. Figure \ref{fig:testmodel}. Here and below, the white color in the plots depicts either perforations or the corresponding values of the field close to zero.}
	\label{fig:solcompare_m1_c}
\end{figure}

\begin{figure}[H]
	\centering
	\subfigure[Reference solution]{
		\includegraphics[trim={4cm 9.5cm 4cm 10cm},clip,width=2.5in]{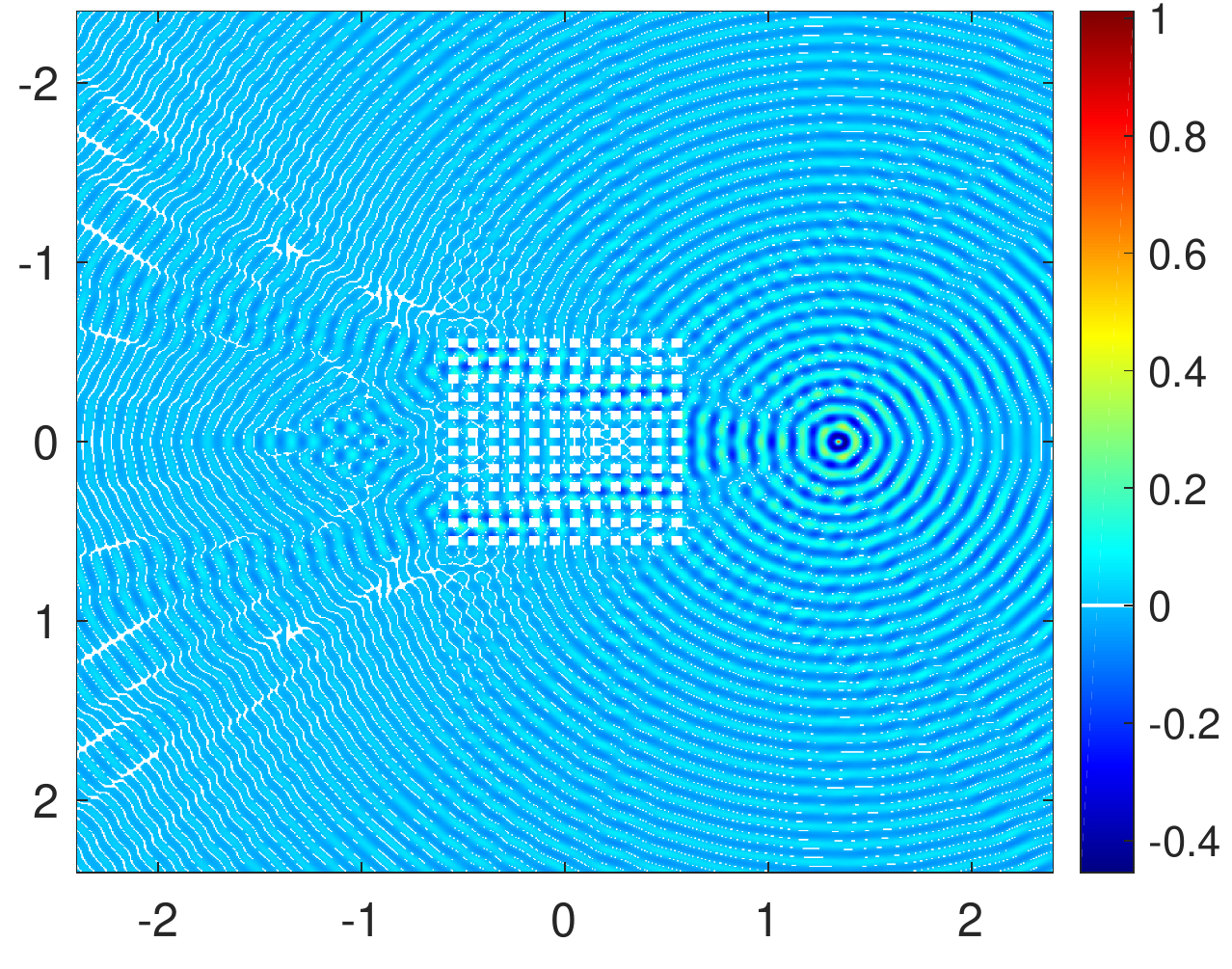}}
	\subfigure[Multiscale solution]{
		\includegraphics[trim={4cm 9.5cm 4cm 10cm},clip,width=2.5in]{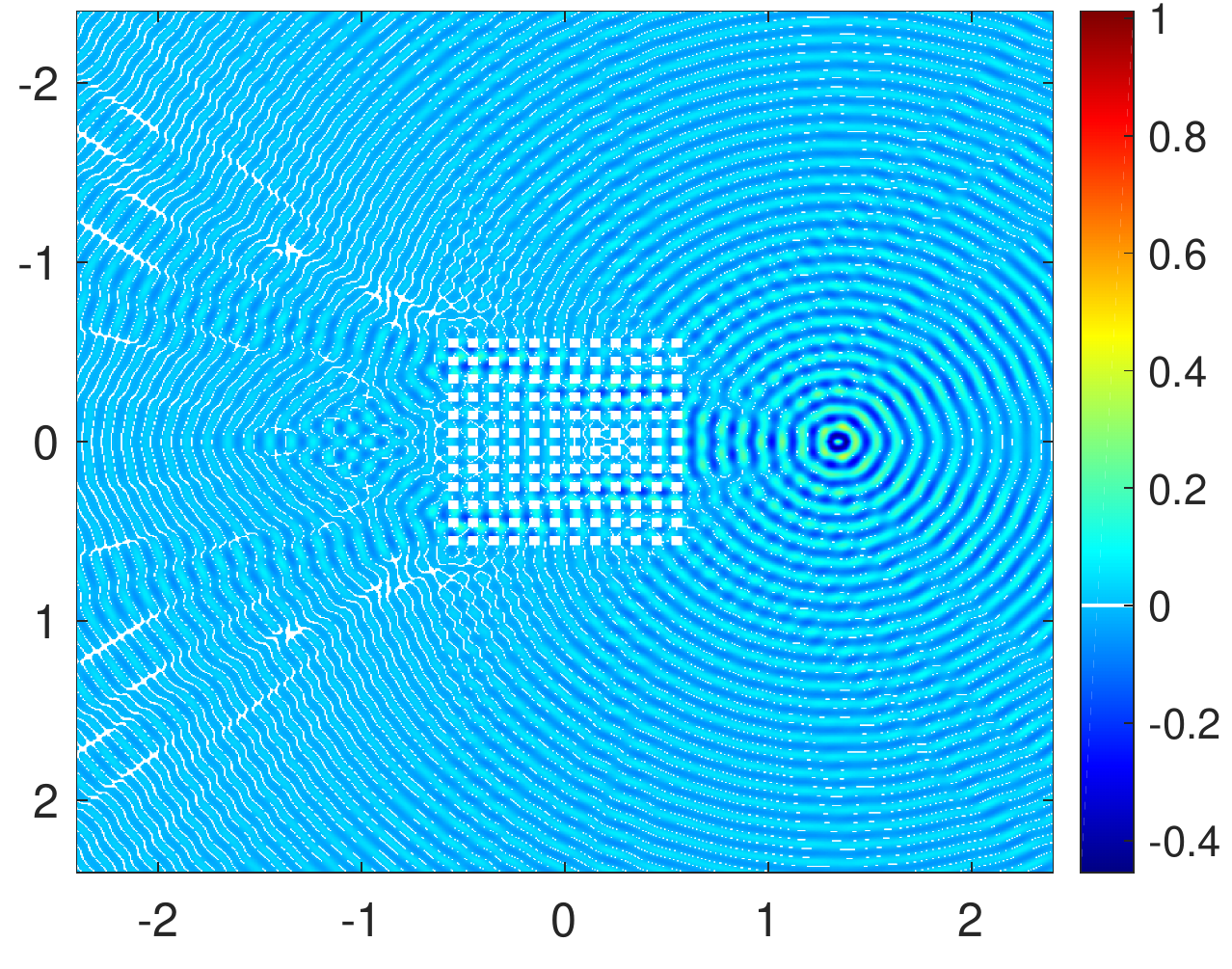}}	
	\caption{Reference solution and multiscale solution for model 1 with right source, $H:=1/10$ and $\ell=2$. The $L^2(\Omega^{\epsilon})$-relative error is 4.24\%. }
	\label{fig:solcompare_m1_r}
\end{figure}

As we can see from Figures \ref{fig:solcompare_m1_c} and
\ref{fig:solcompare_m1_r}, even with $\ell=2$, one can observe the
wave scattering phenomenon resulting from the perforated structure
with clear agreement between the multiscale solution and the reference
solution with full capture of the microscale feature of the wavefield;
we now quantify this agreement. We test the convergence of Algorithm \ref{algorithm:wavelet} with
respect to the level parameter $\ell$ for both of the perforated domains
shown in Figure \ref{fig:testmodel}. To this end, we fix the coarse
scale mesh size $H:=1/10$. Recall that the level parameter $\ell$
determines the number of multiscale basis functions in each coarse
neighborhood $\omega_i$ for $i=1,\cdots,N$ with $N$ as the number of
coarse grids; specifically this number is $2^{\ell+2}+1$ and the level parameter $\ell$ shows the complexity of Algorithm \ref{algorithm:wavelet}.


Figure \ref{fig:er_m1} shows
 the $L^2(\Omega^{\epsilon})$ and $H^1(\Omega^{\epsilon})$-relative
 errors versus $\ell$, and both the $L^2(\Omega^{\epsilon})$ and
 $H^1(\Omega^{\epsilon})$-relative errors decay rapidly as more
 wavelet basis functions are added. For example, for the case that the
 source lies at the center shown in Figure \ref{fig:er_m1}(a), the
 $L^2(\Omega^{\epsilon})$ errors decrease from $120.0\%$ to $4.08\%$
 as the level parameter $\ell$ increases from $0$ to $3$. Figure
 \ref{fig:er_m1} suggests that Algorithm \ref{algorithm:wavelet} with
 level parameter $\ell=2$ yields an accurate solver with little gain
 from going to higher $\ell$.
\begin{figure}[H]
	\centering
	\subfigure[Error history for model 1, centered source.]{
		\includegraphics[trim={1cm 7.5cm 1cm 6.9cm},clip,width=2.5in]{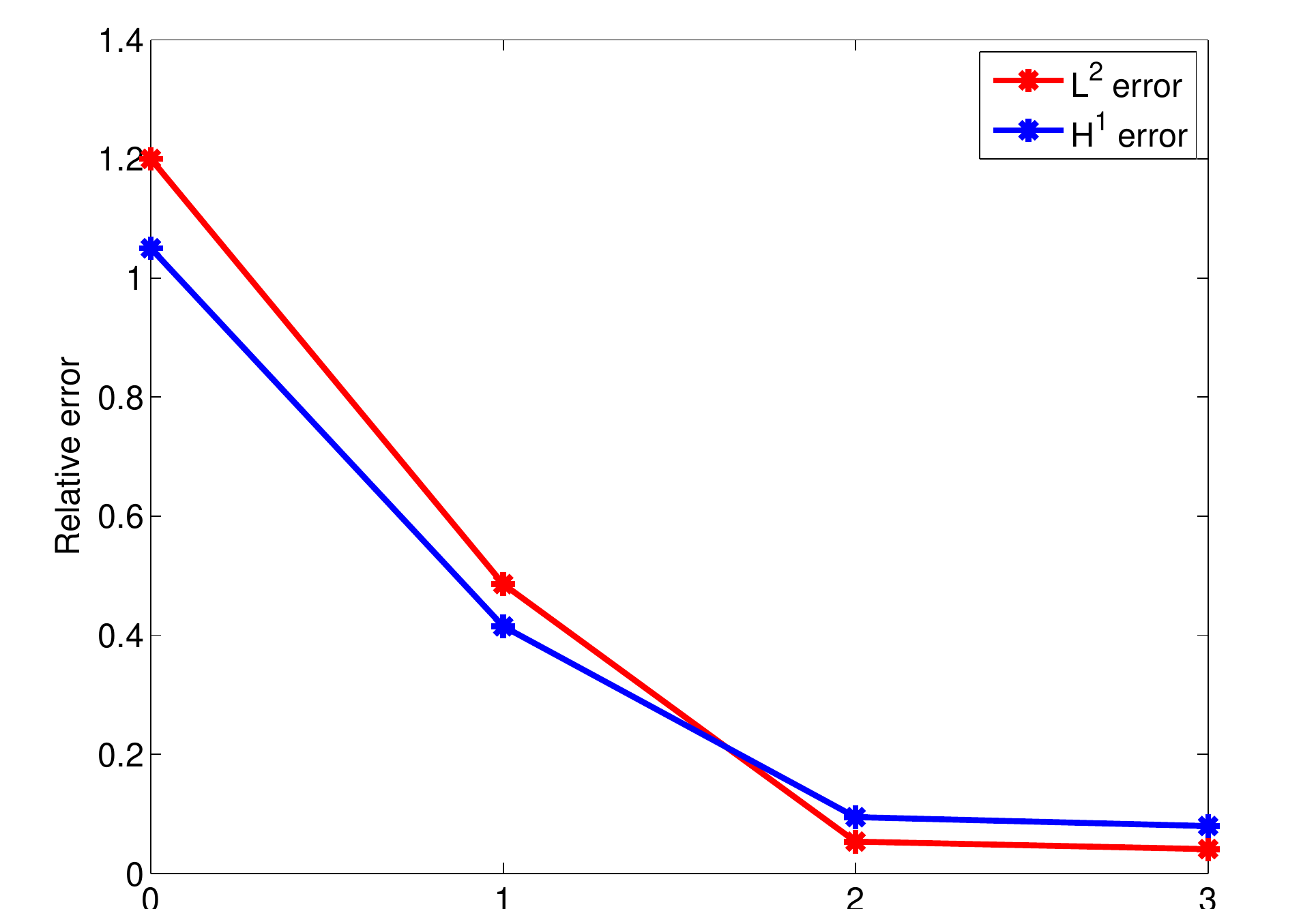}}
	\subfigure[Error history for model 1, right source.]{
		\includegraphics[trim={1cm 7.5cm 1cm 6.9cm},clip,width=2.5in]{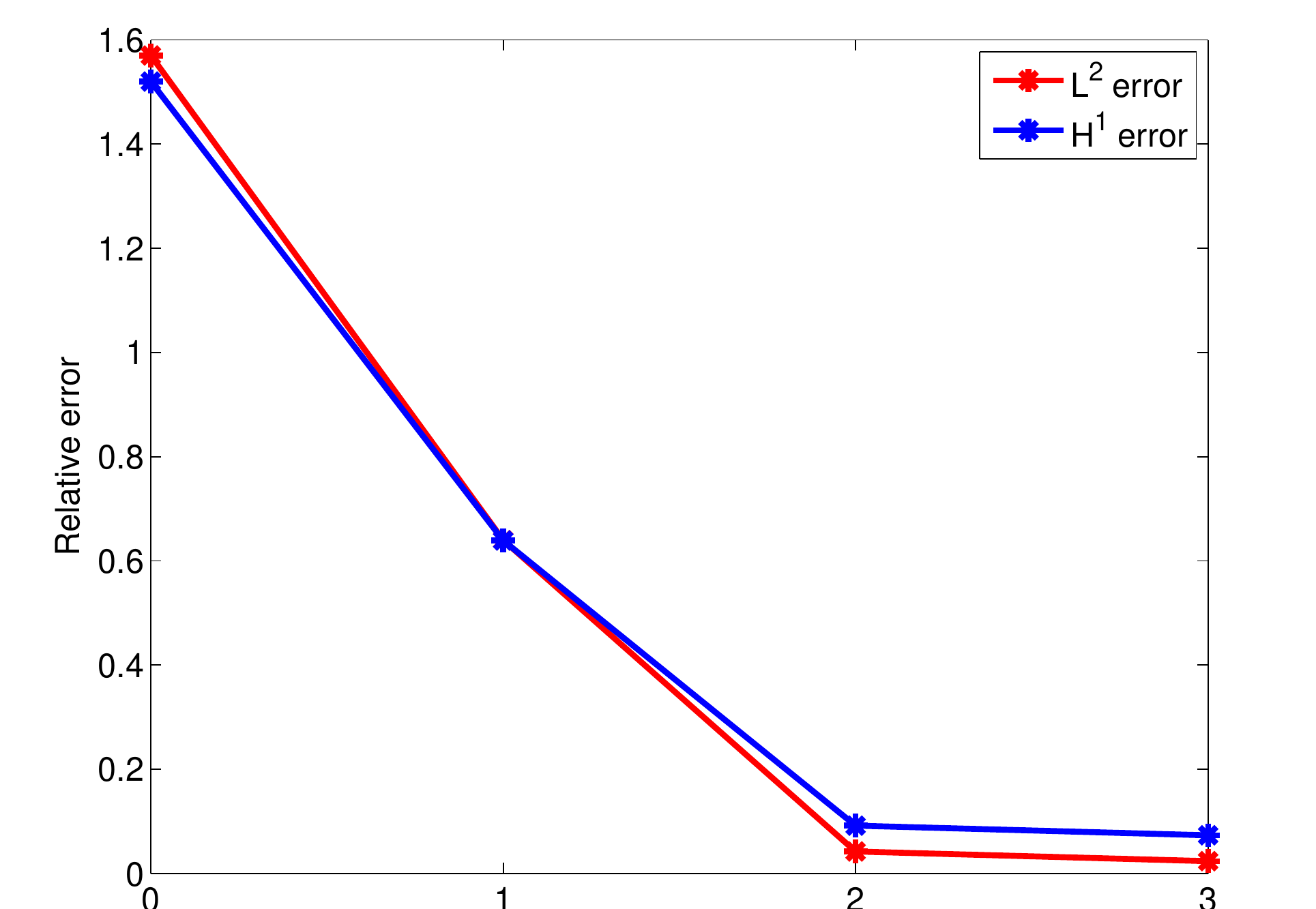}}	
	\caption{Relative error against level, $H=1/10$.}
	\label{fig:er_m1}
\end{figure}
\begin{figure}[H]
	\centering
	\subfigure[Reference solution.]{
		\includegraphics[trim={4cm 9.5cm 4cm 10cm},clip,width=2.5in]{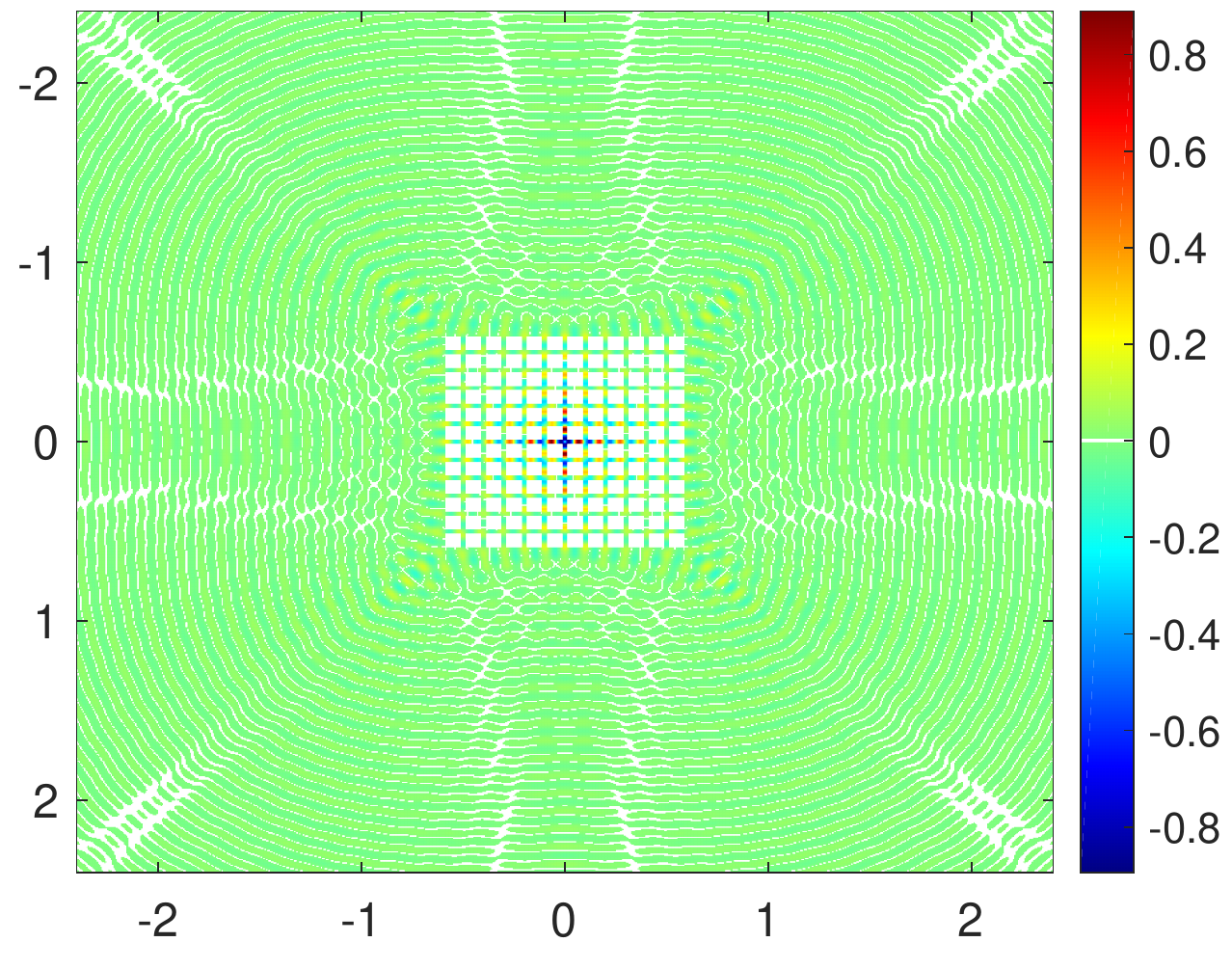}}
	\subfigure[Multiscale solution.]{
		\includegraphics[trim={4cm 9.5cm 4cm 10cm},clip,width=2.5in]{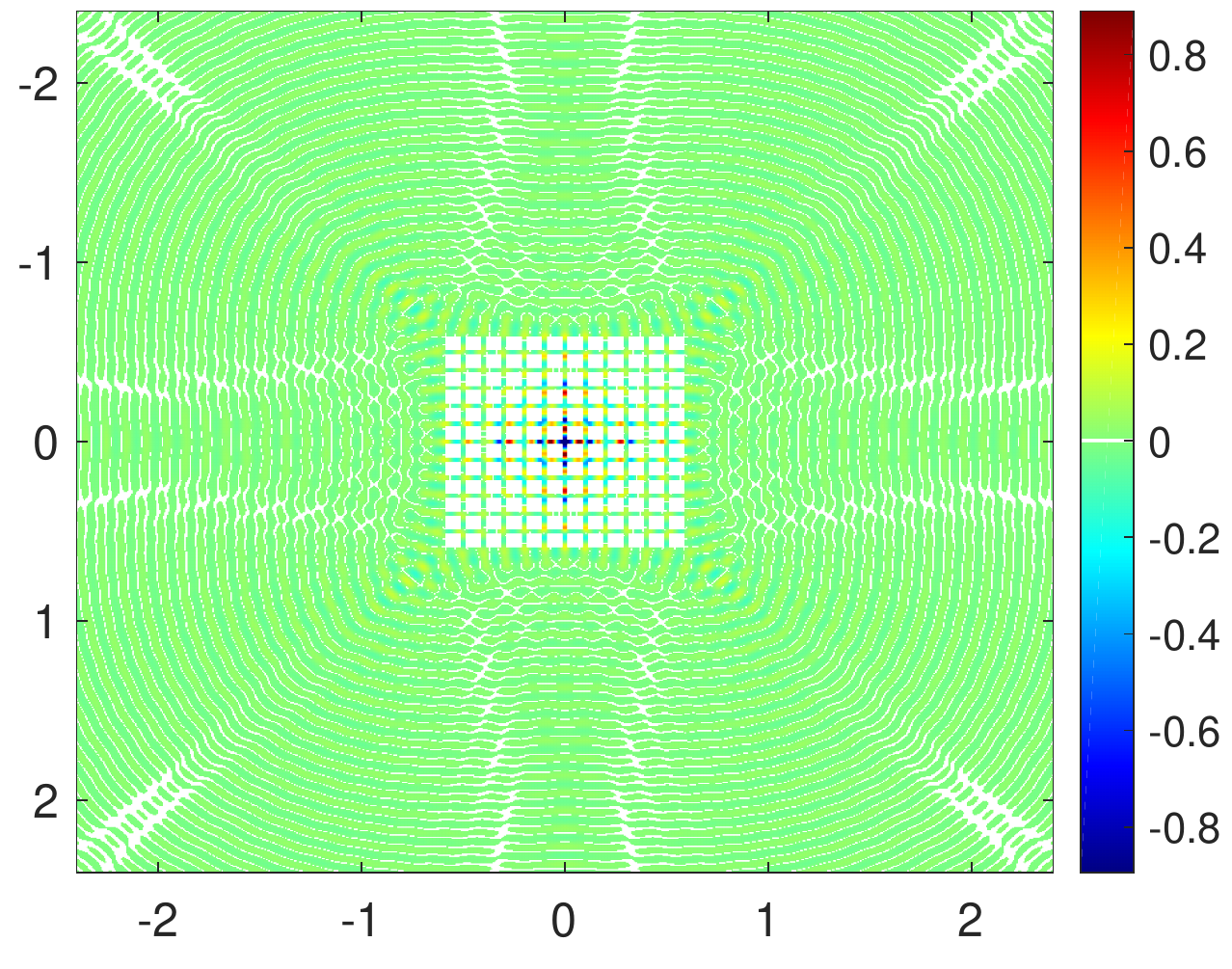}}	
	\caption{Reference solution and multiscale solution for model 2 with centered source, $H:=1/10$ and $\ell=2$. The $L^2(\Omega^{\epsilon})$-relative error is 25.7\%. }
	\label{fig:solcompare_m2_c}
\end{figure}
Similar performance for the perforated domain $\Omega^{\epsilon}$ depicted in
model 2 is obtained, see Figures \ref{fig:solcompare_m2_c} and \ref{fig:solcompare_m2_r}. The second model as compared with the perforated
domain in model 1 has the gap between perforations much narrower, we
expect stronger singularities and finer structure in the solution, and
this provides insight on method robustness. In terms of the physics,
the frequency has remained fixed and altering the perforation size
alters the dynamic anisotropy slightly, we observe strong
directionality and concentration of the highly oscillatory wave field in the narrow
gaps.

 The relative error decay history is shown in Figure \ref{fig:er_m2}
 and we find similar relative error decay behavior as in Figure \ref{fig:er_m1} and Algorithm 1 is both efficient and accurate. For instance, for the case that the source lies
in the center, the $L^2(\Omega^{\epsilon})$-relative error reaches
below $10\%$ even with the level parameter $\ell=1$.

\begin{figure}[H]
	\centering
	\subfigure[Reference solution.]{
		\includegraphics[trim={4cm 9.5cm 4cm 10cm},clip,width=2.5in]{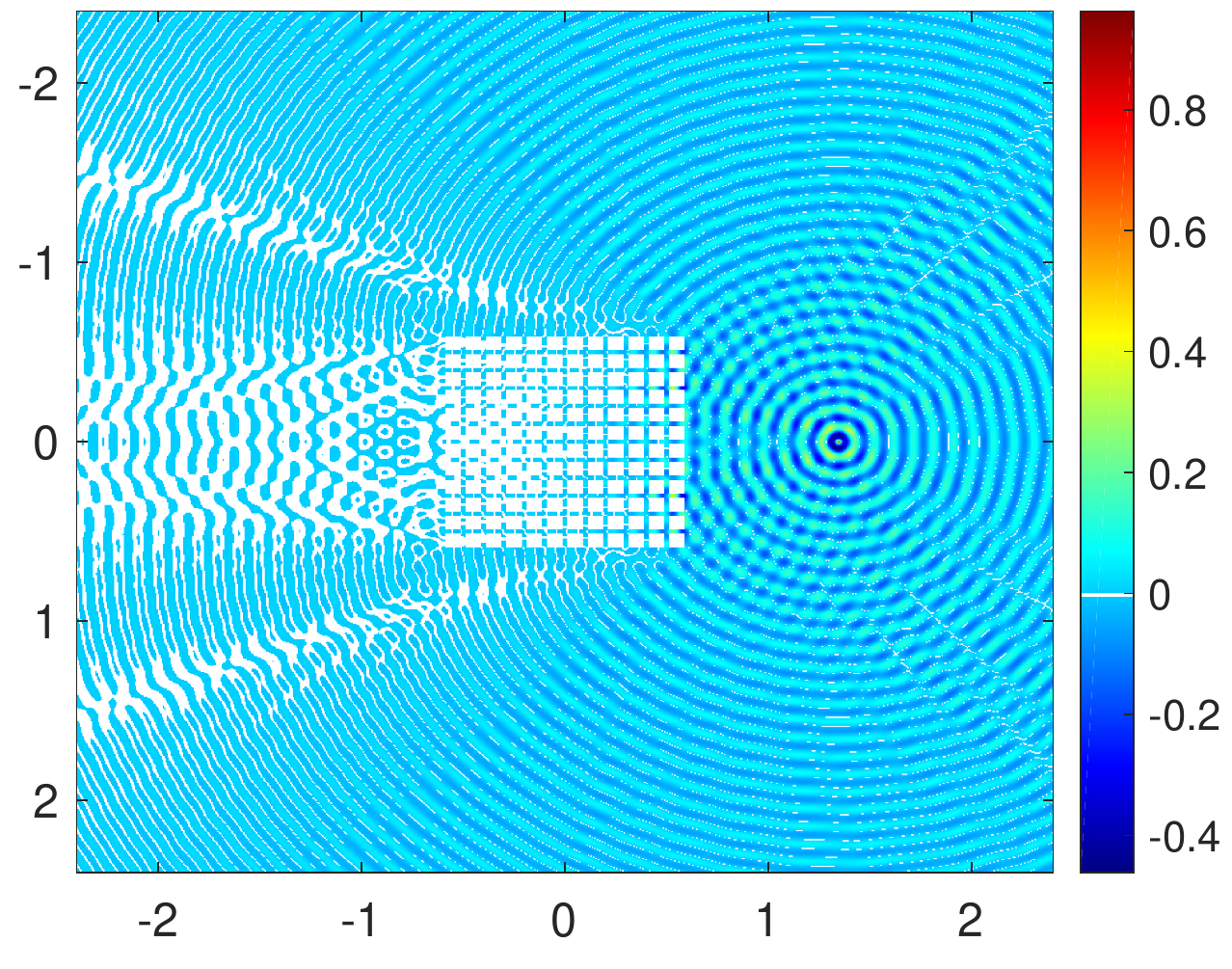}}
	\subfigure[Multiscale solution.]{
		\includegraphics[trim={4cm 9.5cm 4cm 10cm},clip,width=2.5in]{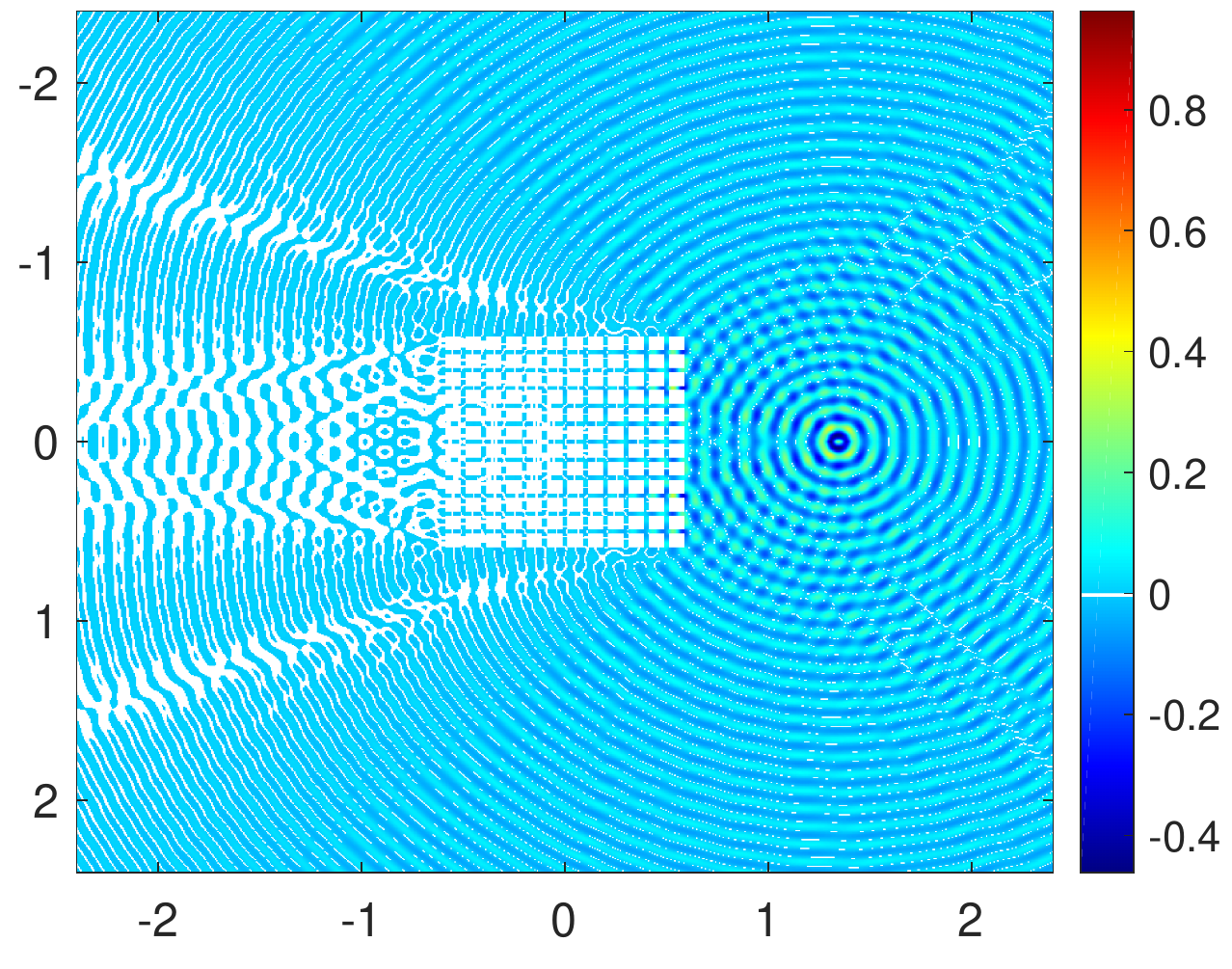}}	
	\caption{Reference solution and multiscale solution for model 2 with right source, $H:=1/10$ and $\ell=2$. The $L^2(\Omega^{\epsilon})$-relative error is 4.20\%.}
	\label{fig:solcompare_m2_r}
\end{figure}

%
%

\begin{figure}[H]
	\centering
	\subfigure[Error for model 2, centered source.]{
		\includegraphics[trim={1cm 7.5cm 1cm 7cm},clip,width=2.5in]{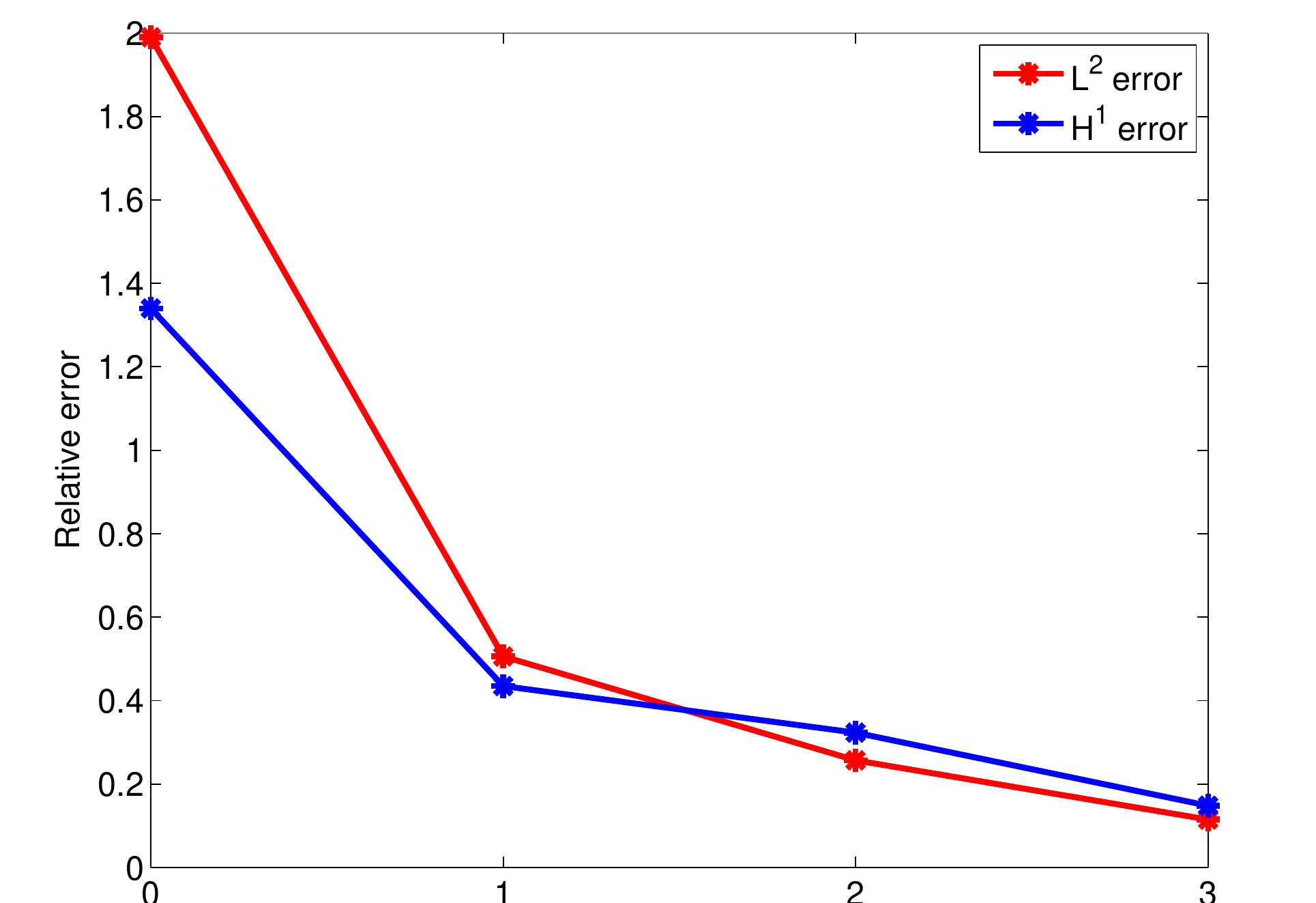}}
	\subfigure[Error for model 2, right source. ]{
		\includegraphics[trim={1cm 7.5cm 1cm 7cm},clip,width=2.5in]{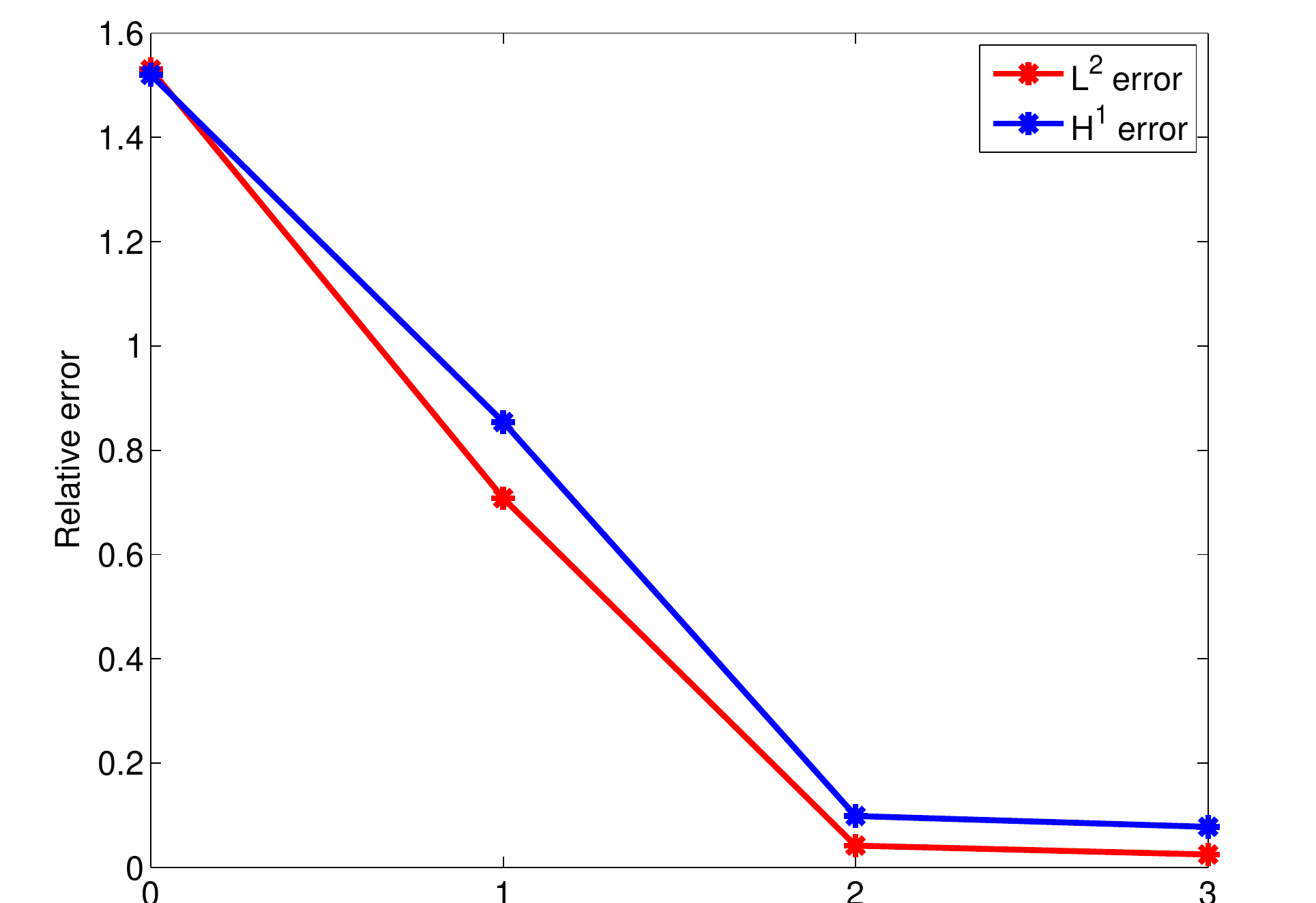}}	
	\caption{Relative error against level, $H=1/10$.}
	\label{fig:er_m2}
\end{figure}

\subsection{Performance of Algorithm \ref{algorithm:wavelet}:
  coarse-scale mesh size $H$}\label{subsec:meshsize}
Earlier we established theoretically that the
error induced by Algorithm \ref{algorithm:wavelet} can attain
$\mathcal{O}(H)$ in Proposition \ref{prop:wavelet-basedconv}, upon the
condition on the coarse mesh size $H$ and the level parameter $\ell$,
cf. \eqref{eq:wavelet-level}.  We now test how the algorithm
performs with respect to a different, finer, coarse-scale mesh, we take its
size $H:=1/20$, so that the coarse mesh $\mathcal{T}_H$ will cross the perforations.

The error decay history for the two different perforated domains, depicted in Figure \ref{fig:testmodel} with centered and right sources, are plotted in Figures \ref{fig:er_m1_H20} and
\ref{fig:er_m2_H20}, respectively. There is error
decay as the level parameter $\ell$ with very rapid decay and we
conclude that multiscale solutions with sufficient accuracy are achieved with the level parameter $\ell=2$.
Comparing with the mesh of $H:=1/10$, Figures \ref{fig:er_m1} and
\ref{fig:er_m2}, we conclude that as the coarse-scale mesh size $H$
decreases, the performance of Algorithm \ref{algorithm:wavelet}
significantly improves; this agrees with the predictions of Proposition \ref{prop:wavelet-basedconv}.
\begin{figure}[H]
	\centering
	\subfigure[Error for model 1, center source.]{
		\includegraphics[trim={1cm 7.5cm 1cm 7cm},clip,width=2.5in]{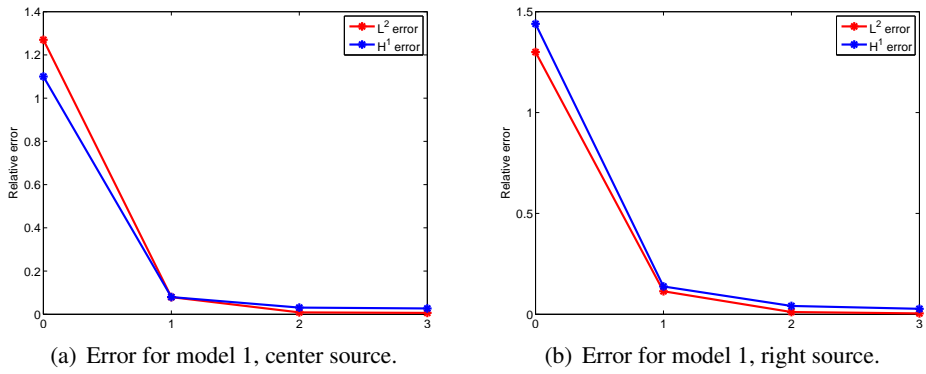}}
	\subfigure[Error for model 1, right source. ]{
		\includegraphics[trim={1cm 7.5cm 1cm 7cm},clip,width=2.5in]{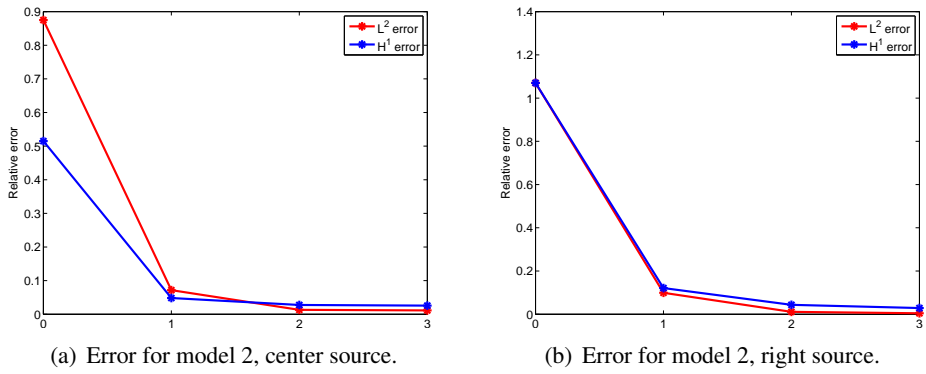}}	
	\caption{Relative error against level, $H=1/20$.}
	\label{fig:er_m1_H20}
\end{figure}

\begin{figure}[H]
	\centering
	\subfigure[Error for model 2, center source.]{
		\includegraphics[trim={1cm 7.5cm 1cm 7cm},clip,width=2.5in]{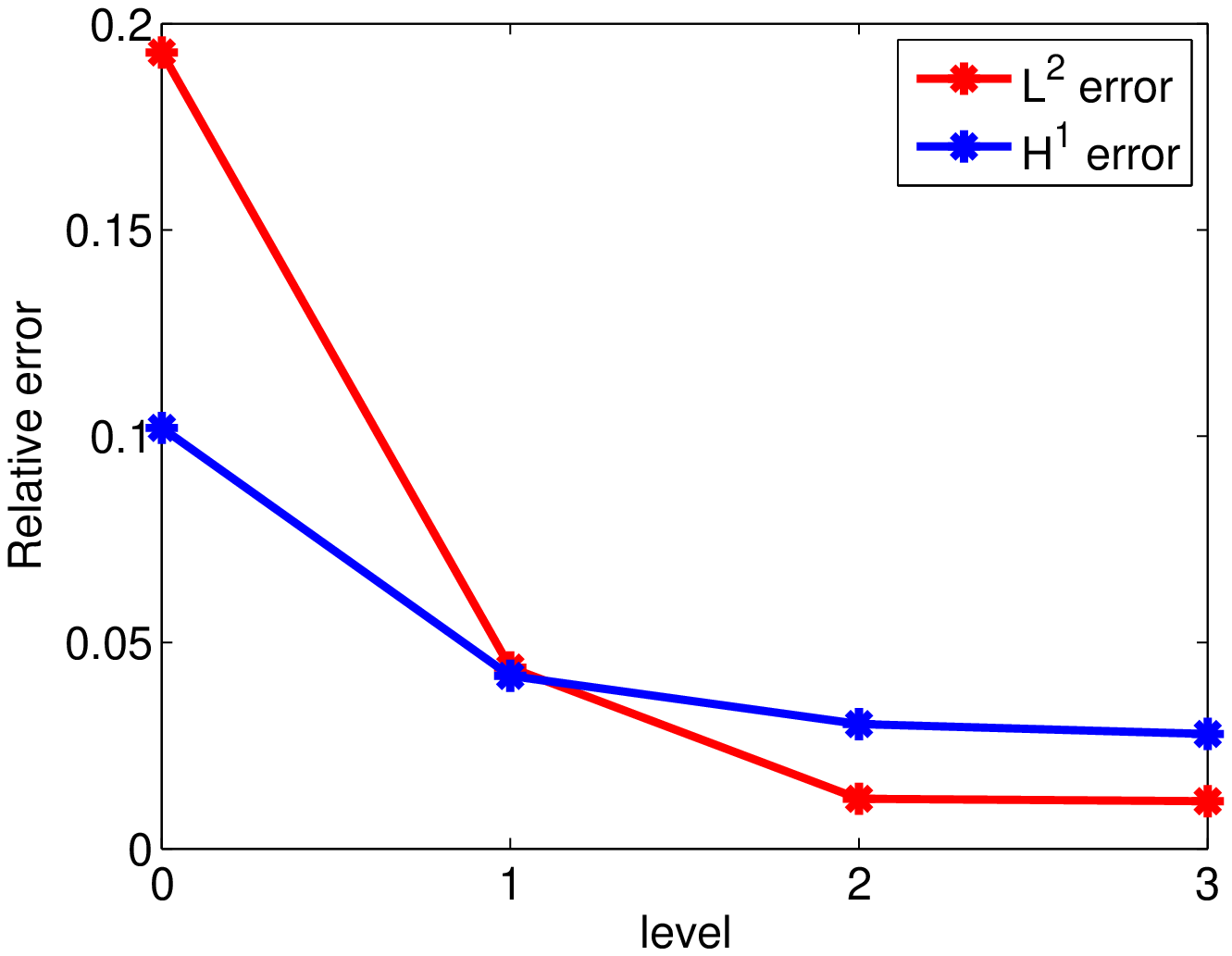}}
	\subfigure[Error for model 2, right source. ]{
		\includegraphics[trim={1cm 7.5cm 1cm 7cm},clip,width=2.5in]{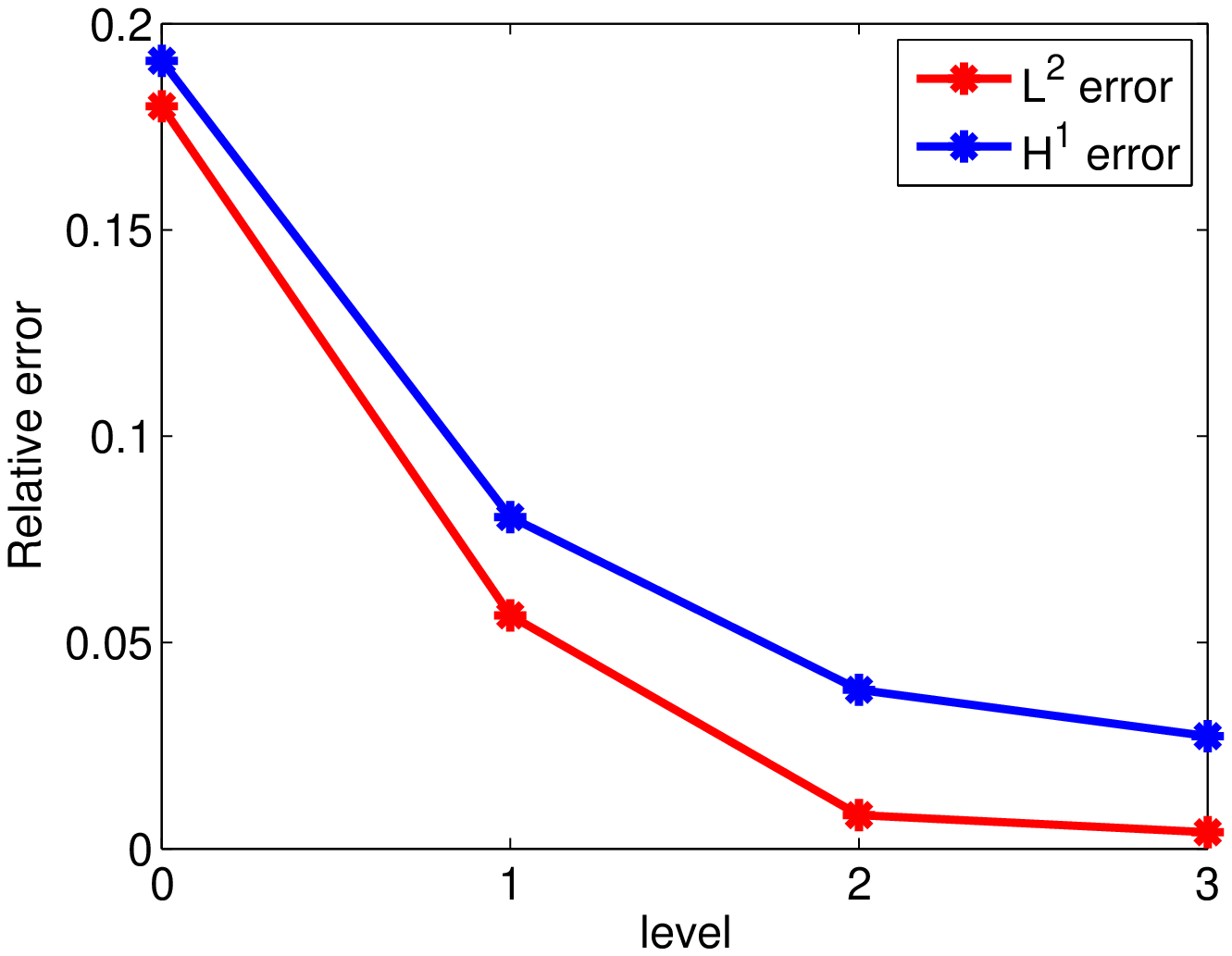}}	
	\caption{Relative error against level, $H=1/20$.}
	\label{fig:er_m2_H20}
\end{figure}

\subsection{Performance of Algorithm \ref{algorithm:wavelet}: random
  perforations}\label{subsec:random}
\begin{figure}[H]
	\centering
	\subfigure[model 3]{
		\includegraphics[trim={4cm 9.5cm 4cm 10cm},clip,width=2.5in]{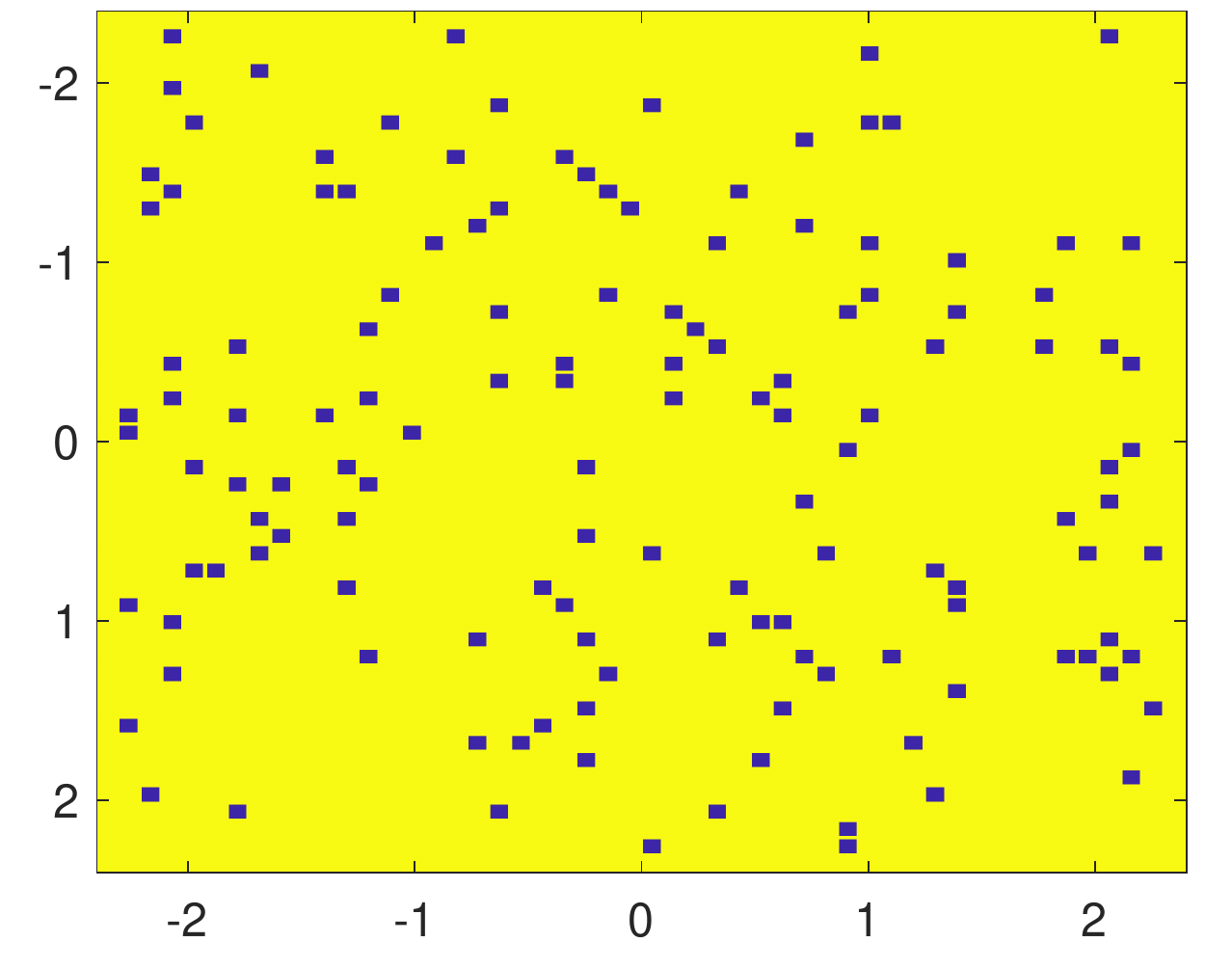}}
	\subfigure[model 4]{
		\includegraphics[trim={4cm 9.5cm 4cm 10cm},clip,width=2.5in]{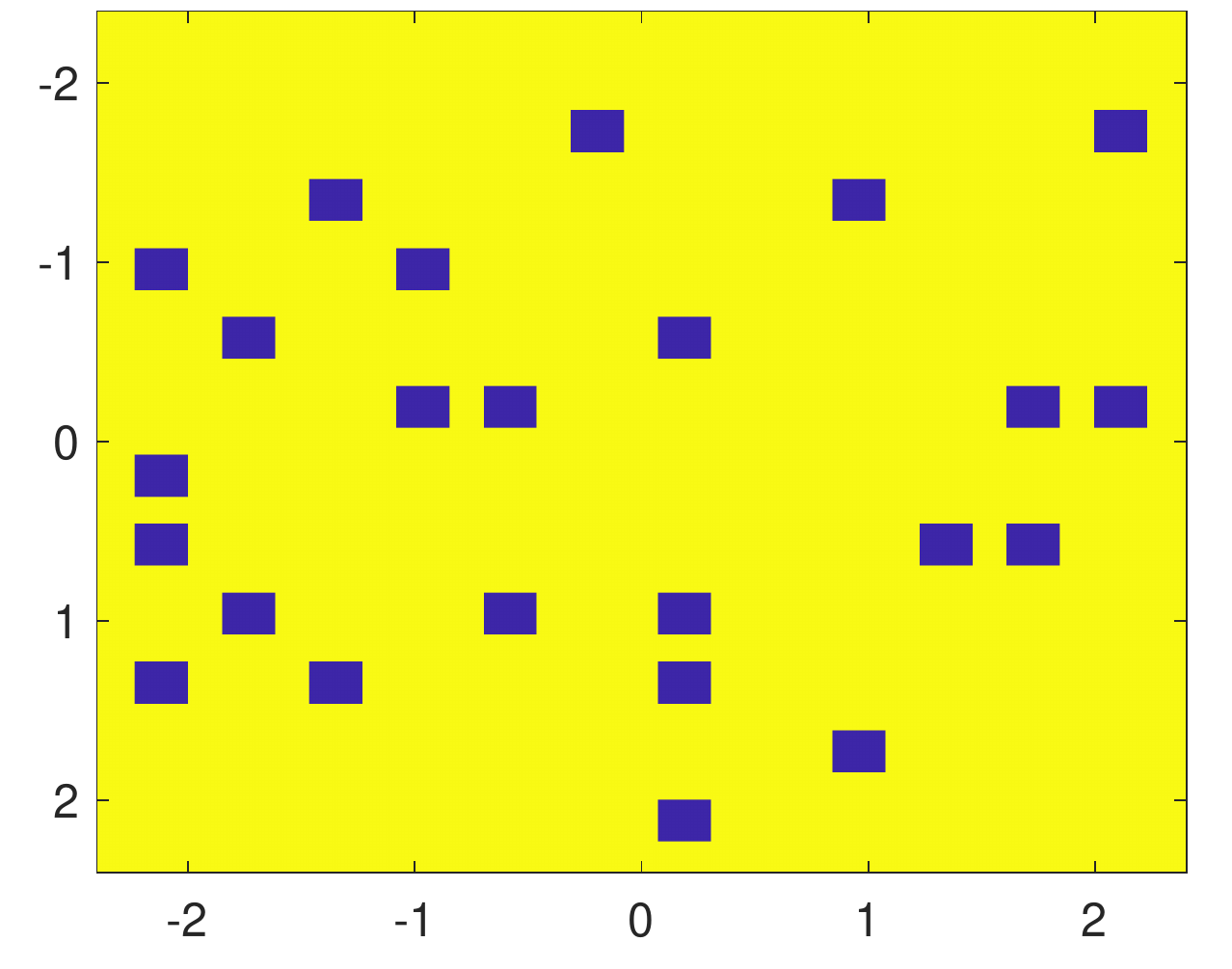}}	
	\caption{Perforated domains, models of locally non-periodic
		crystals}
	\label{fig:testmodel34}
\end{figure}
The motivation for the development of the multiscale WEMsFEM was to
emerging problems in finite periodic crystals, but the methodology is
not reliant on any periodicity and we investigate the algorithm's
performance in more general situations. A natural case to consider is
that of random perforated domains as in Figure
\ref{fig:testmodel34}. The size of the perforation for model 3 is
0.08, and for model 4 is 0.24 and we consider $H=1/20 \text{ and } 1/10$ for model 3 and $H=1/10$ for model 4.

Firstly, we present the reference solution and multiscale solution
from Algorithm \ref{algorithm:wavelet} with model 3 as the perforated
domain, a centered source term, a coarse mesh size $H:=1/10$ and level
parameter $\ell=2$ in Figure \ref{fig:solcompare_rand_c}. The
$L^2(\Omega^{\epsilon})$-relative error is 3.70\%; further decreasing
the coarse mesh size $H$ or increasing the level parameter $\ell$
improves the accuracy as shown in Figure \ref{fig:er_m3}.

\begin{figure}[H]
	\centering
	\subfigure[Reference solution.]{
		\includegraphics[trim={4cm 9.5cm 4cm 10cm},clip,width=2.5in]{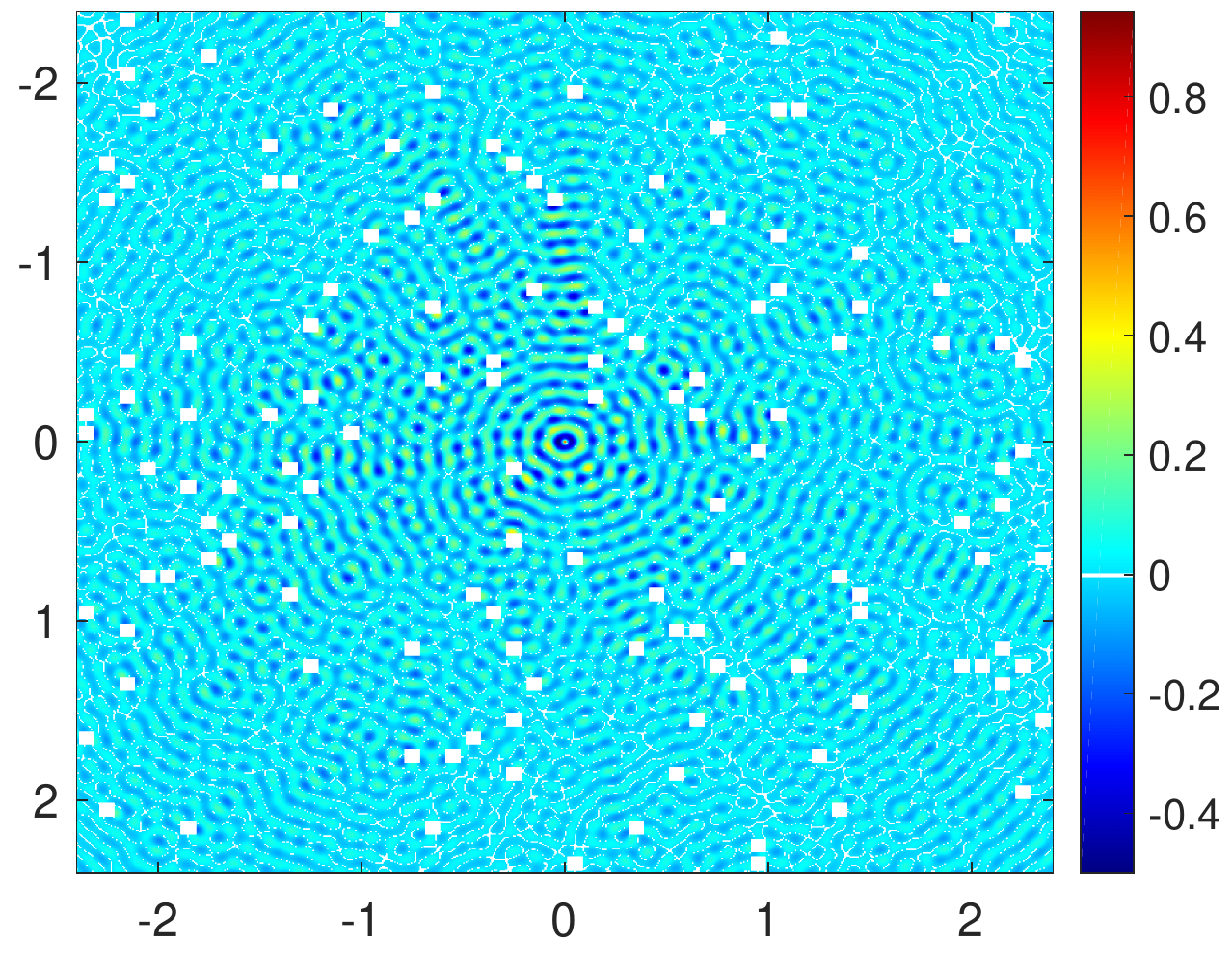}}
	\subfigure[Multiscale solution.]{
		\includegraphics[trim={4cm 9.5cm 4cm 10cm},clip,width=2.5in]{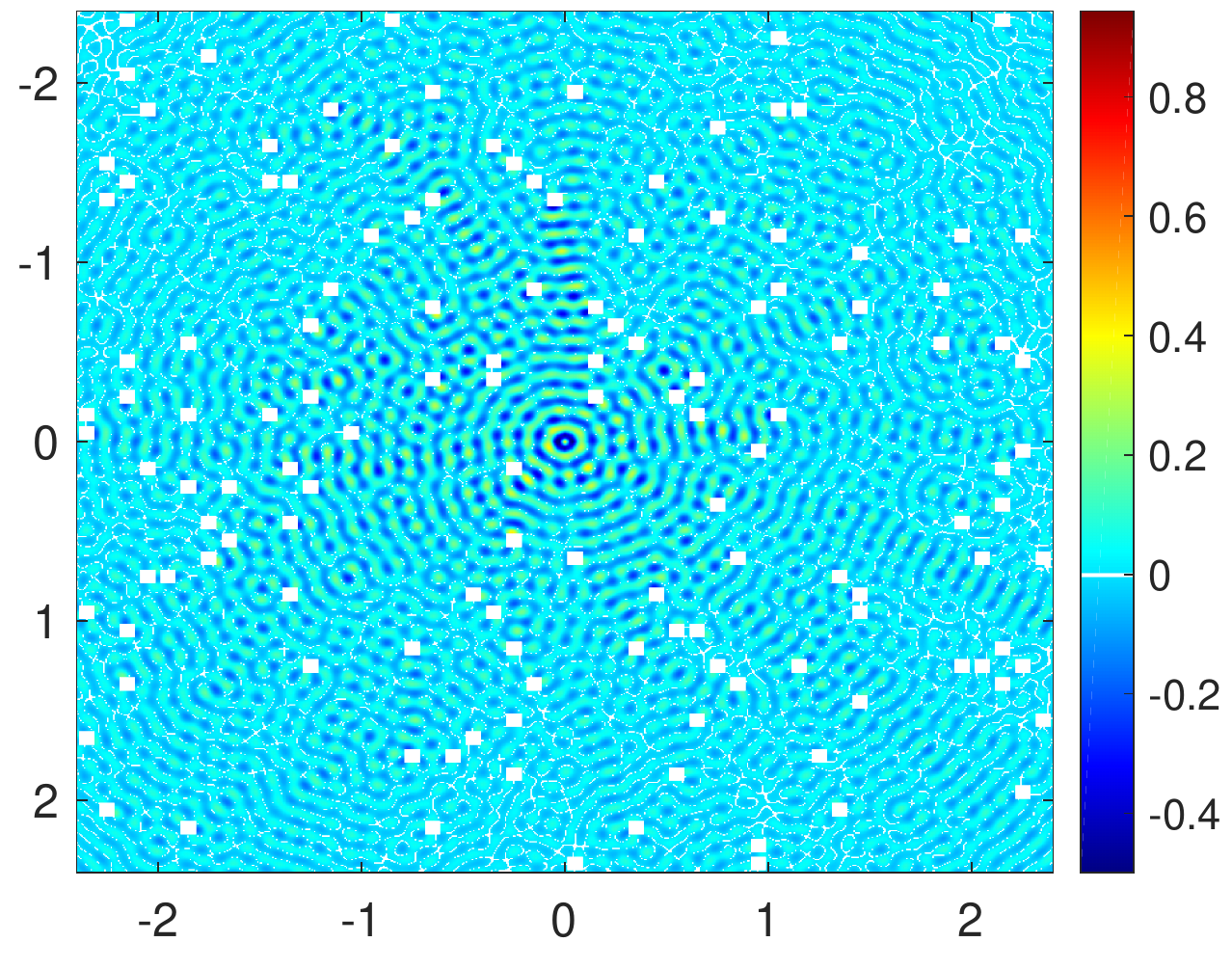}}	
	\caption{Reference solution and multiscale solution for model 3 with centered source, $H:=1/10$ and $\ell=2$. The $L^2(\Omega^{\epsilon})$-relative error is 3.70\%. }
	\label{fig:solcompare_rand_c}
\end{figure}

\begin{figure}[H]
	\centering
	\subfigure[Error for model 3, $H=1/10$.]{
		\includegraphics[trim={1cm 7.5cm 1cm 7cm},clip,width=2.5in]{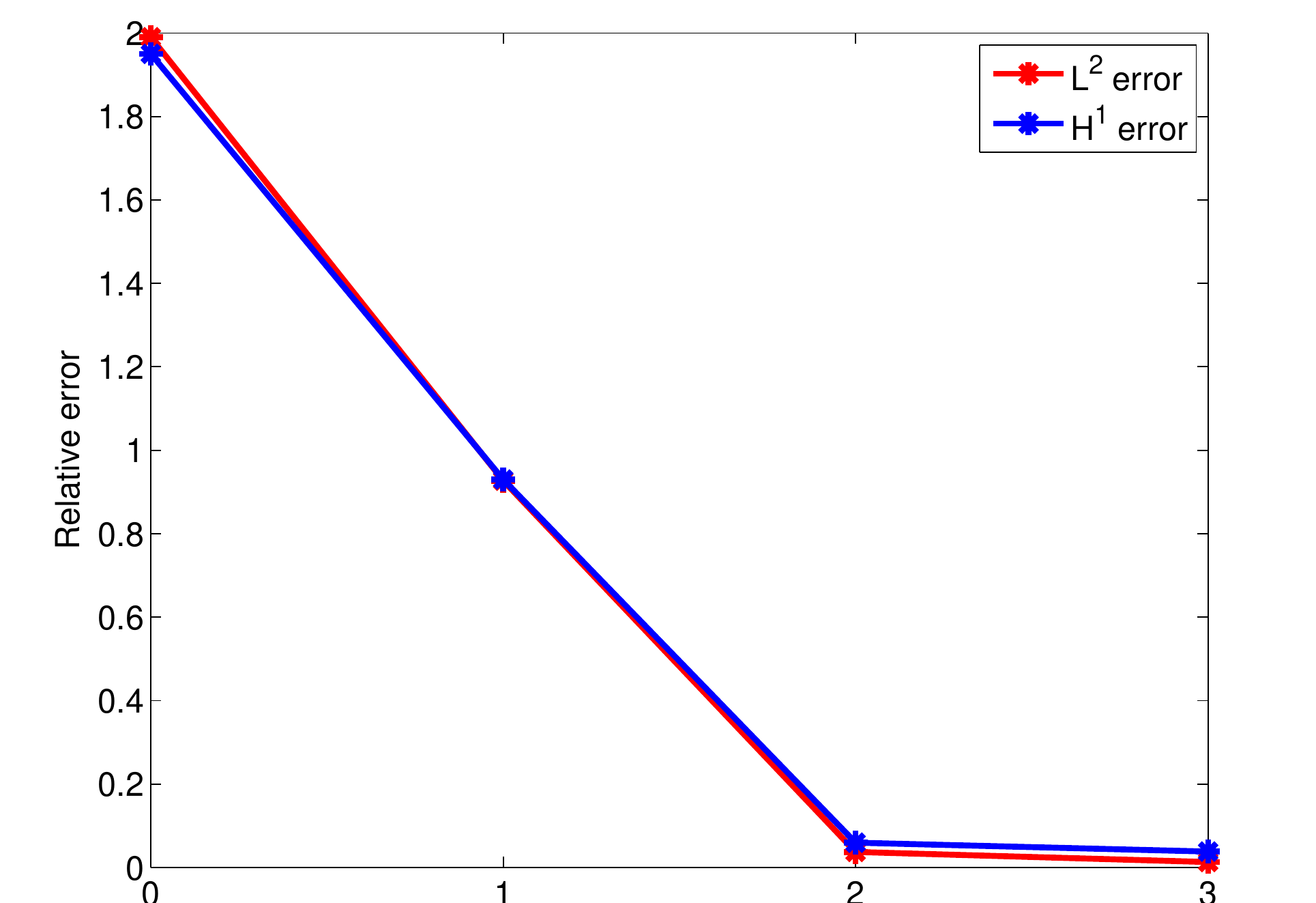}}
	\subfigure[Error for model 3, $H=1/20$. ]{
		\includegraphics[trim={1cm 7.5cm 1cm 7cm},clip,width=2.5in]{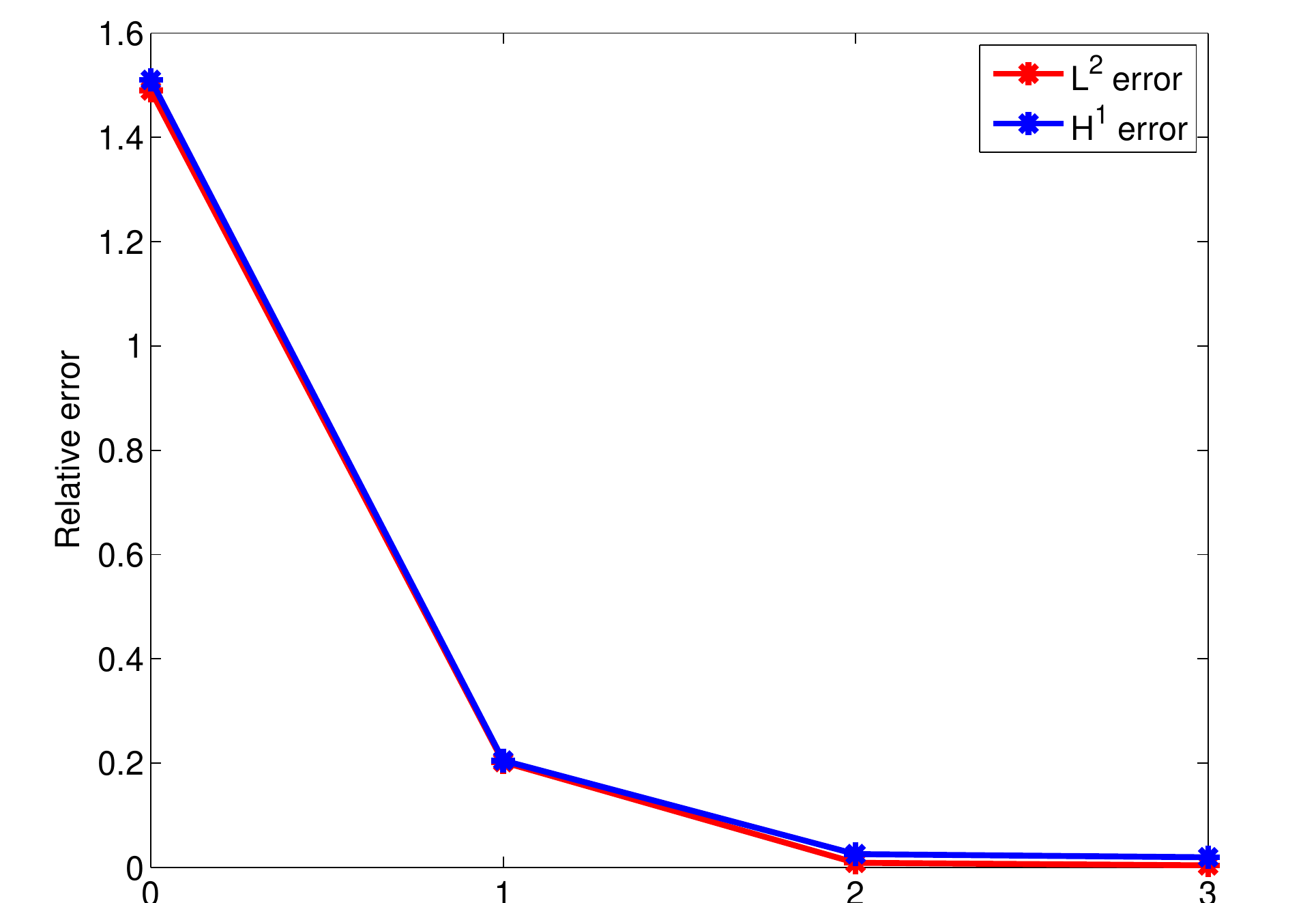}}	
	\caption{Relative error against level, $H=1/20$.}
	\label{fig:er_m3}
\end{figure}

Next, we study the performance of Algorithm \ref{algorithm:wavelet} in
the perforated domain of model 4. The perforations in model 4 cross the neighborhood
boundary and we depict the resulting four coarse neighborhood
$\omega_i$ in Figure \ref{fig:localmodel4}. The local solvers in
Algorithm \ref{algorithm:wavelet} are now defined in some L-shaped
domains and consequently, the perforated domain of model 4 is much
more challenging numerically. 

\begin{figure}[H]
\centering
\begin{tabular}{c c}
	\subfigure{
		\includegraphics[trim={4cm 9.5cm 4cm 10cm},clip,width=2.5in]{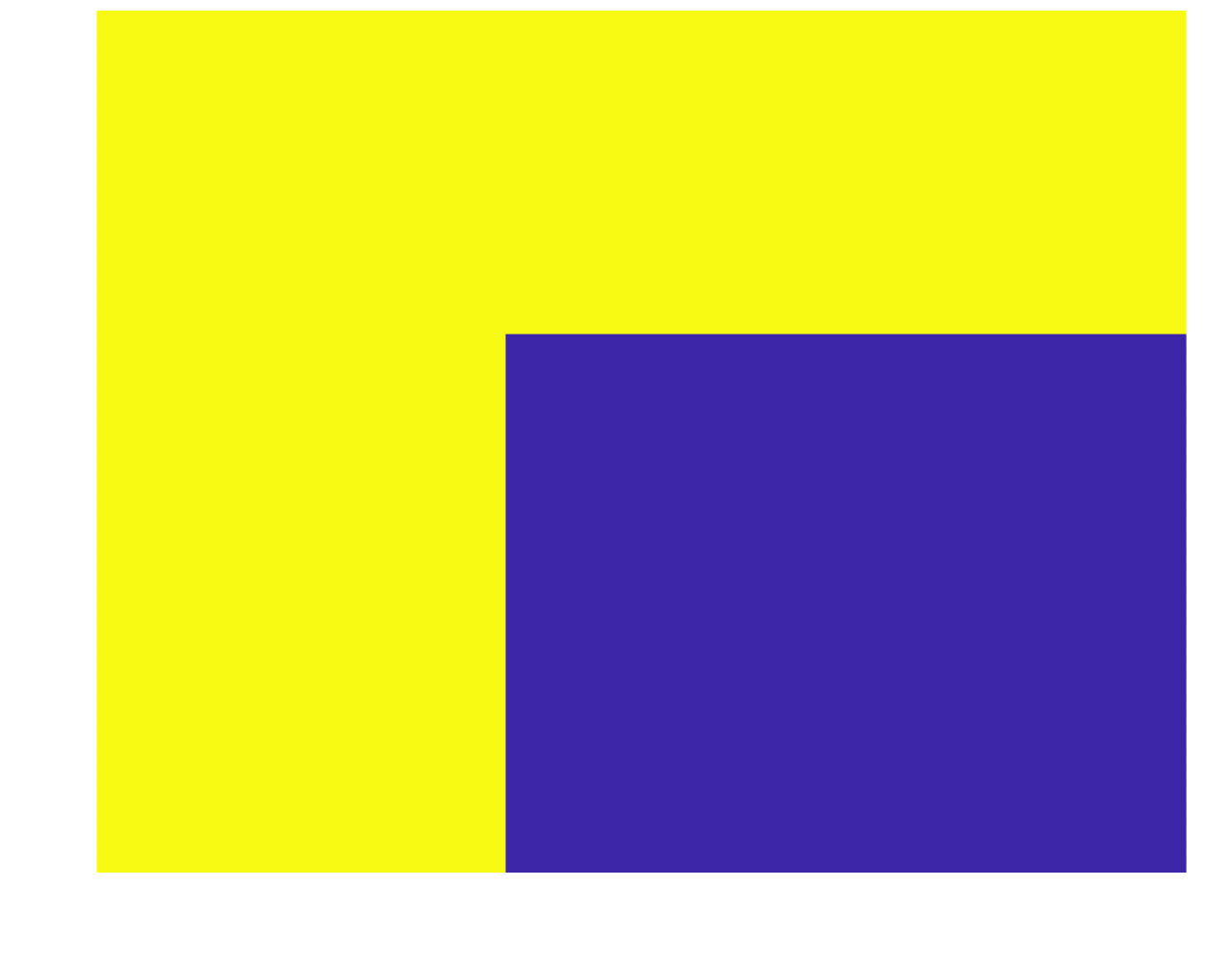}}&
	\subfigure{
		\includegraphics[trim={4cm 9.5cm 4cm 10cm},clip,width=2.5in]{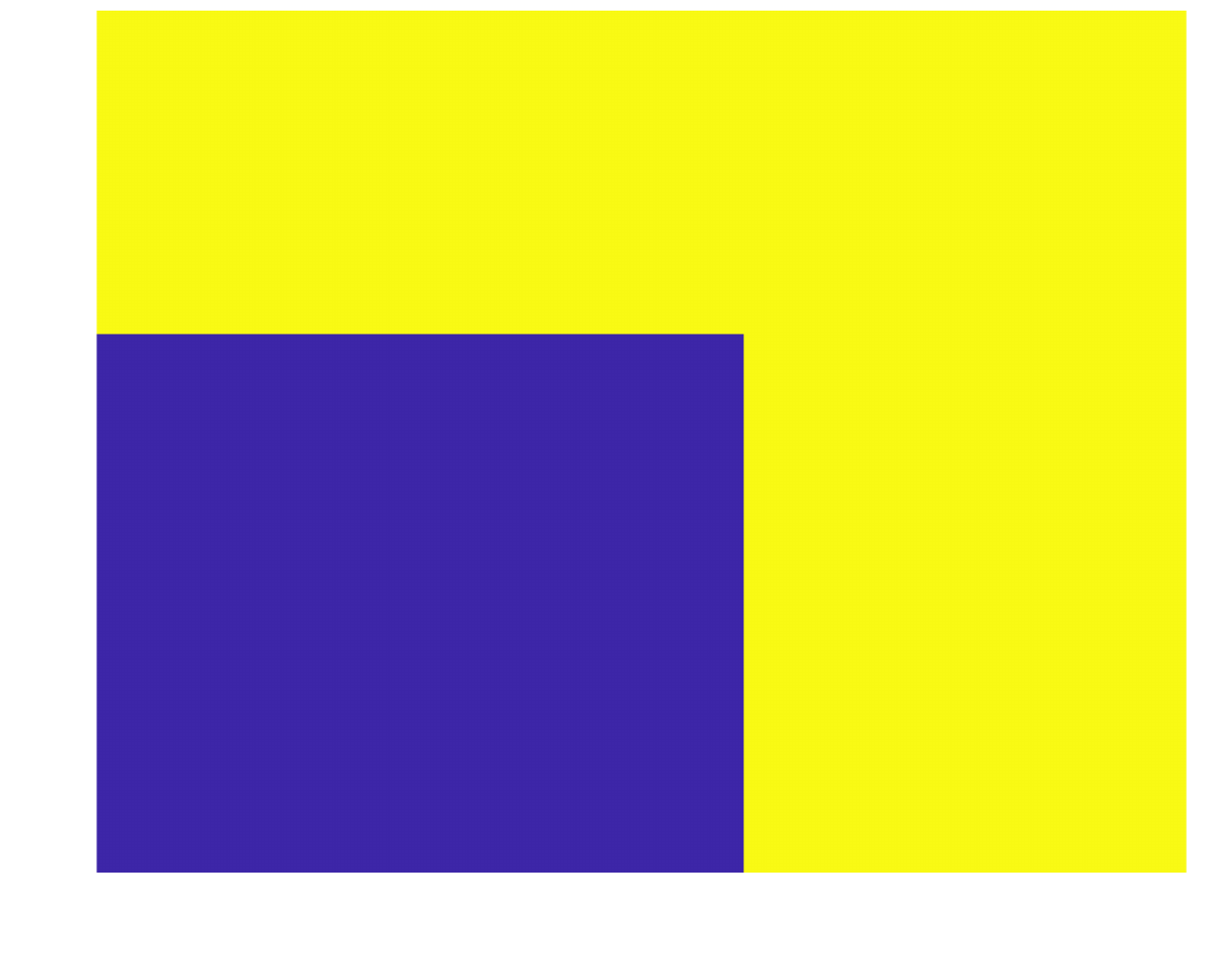}}	\\
	\subfigure{
	\includegraphics[trim={4cm 9.5cm 4cm 10cm},clip,width=2.5in]{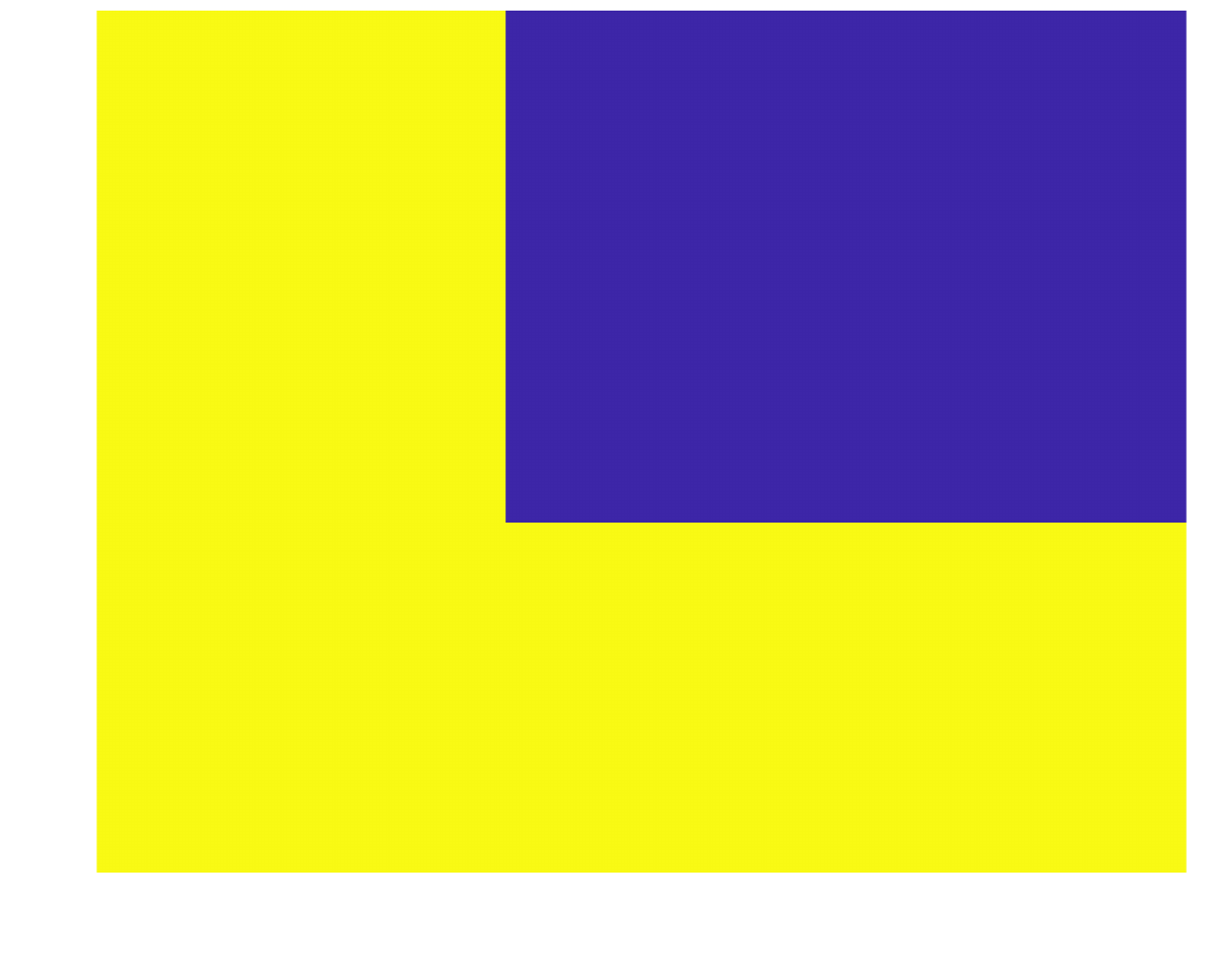}}&
\subfigure{
	\includegraphics[trim={4cm 9.5cm 4cm 10cm},clip,width=2.5in]{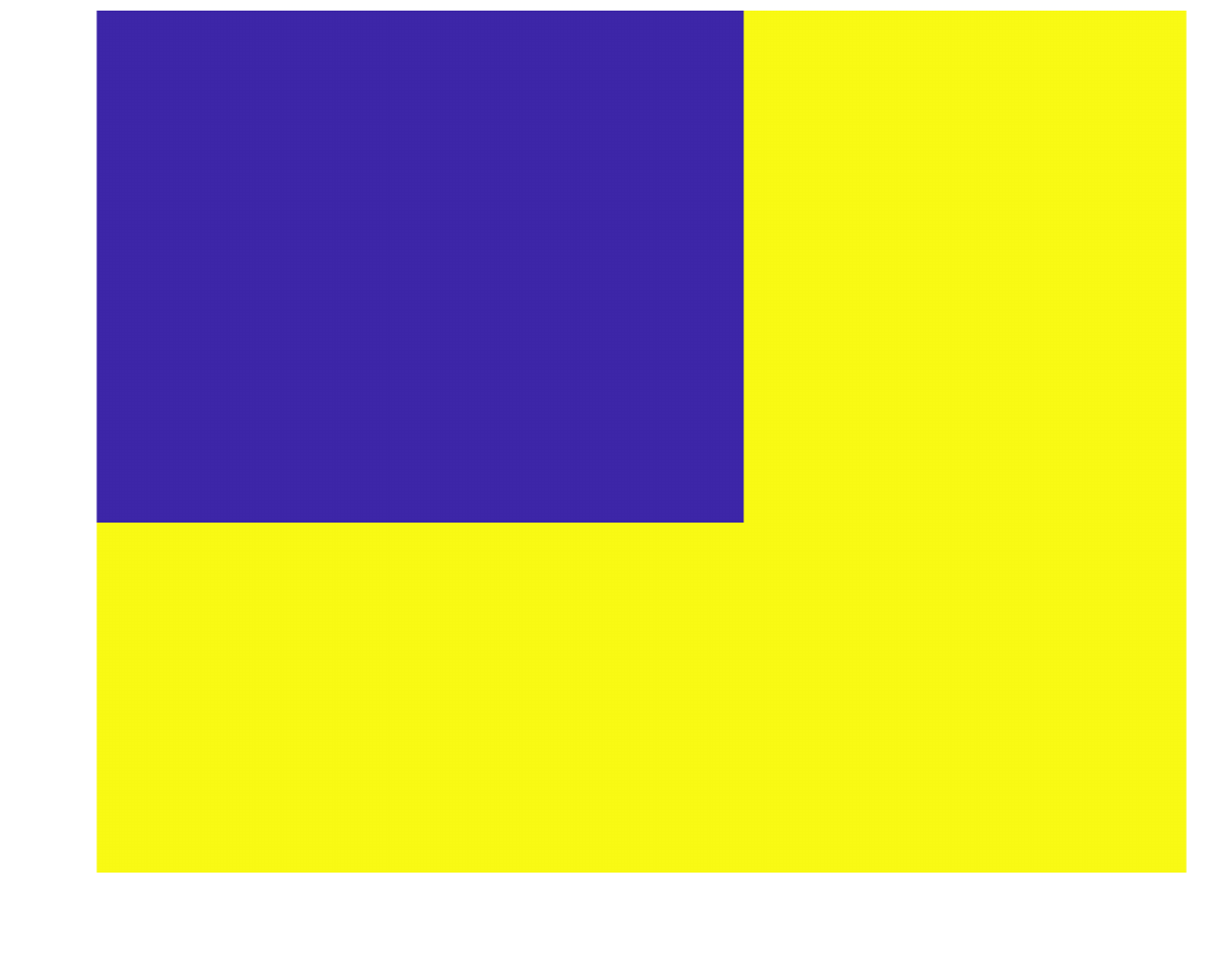}}	
	\end{tabular}
	\caption{Examples of  neighborhoods with local perforation in model 4}
	\label{fig:localmodel4}
\end{figure}

Nonetheless, we observe similar performance as for model 3, with a
centered source, $H:=1/10$ and $\ell=2$ we show the results in in
Figure \ref{fig:solcompare_rand_largec}; the corresponding
$L^2(\Omega^{\epsilon})$-relative error is 3.82\%. Analogously to
model 3, further increases in the level parameter $\ell$, or decrease
in the coarse grid size $H$, further improve the performance of our algorithm, see Figure \ref{fig:er_m4} for more details.

\begin{figure}[H]
	\centering
	\subfigure[Reference solution.]{
		\includegraphics[trim={4cm 9.5cm 4cm 10cm},clip,width=2.5in]{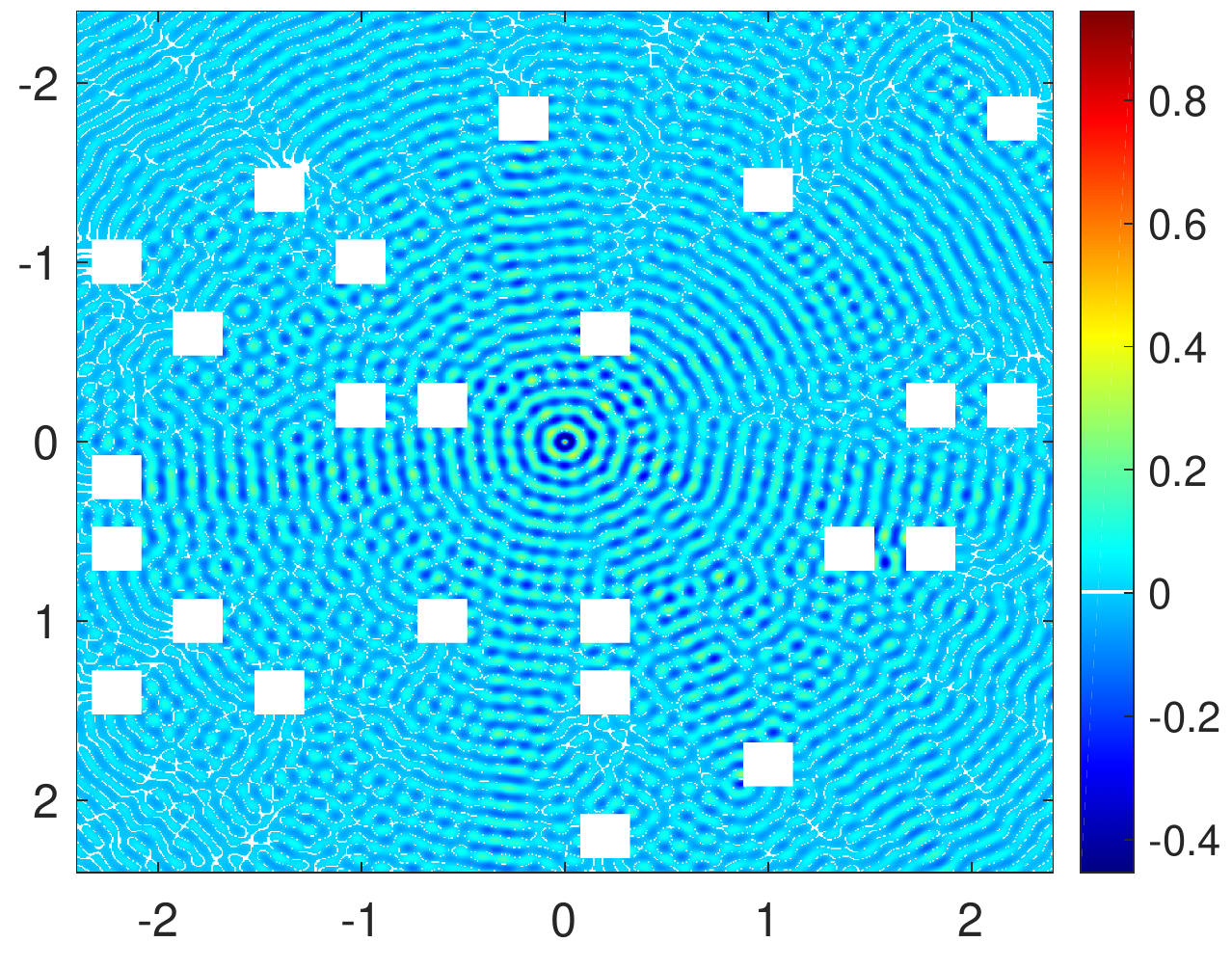}}
	\subfigure[Multiscale solution.]{
		\includegraphics[trim={4cm 9.5cm 4cm 10cm},clip,width=2.5in]{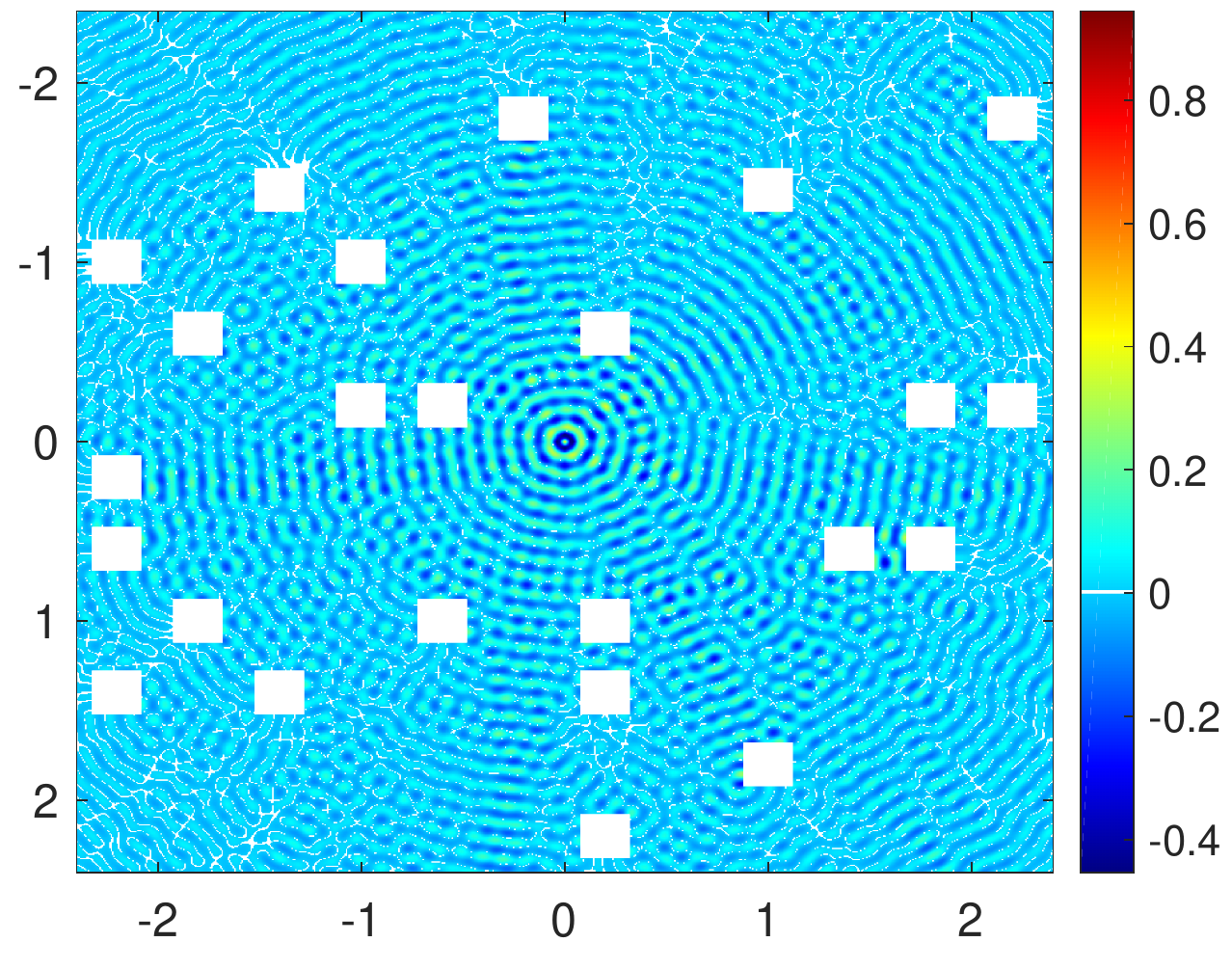}}	
	\caption{Reference solution and multiscale solution for model 4 with centered source, $H:=1/10$ and $\ell=2$. The $L^2(\Omega^{\epsilon})$-relative error is 3.82\%. }
	\label{fig:solcompare_rand_largec}
\end{figure}

\begin{figure}[H]
	\centering
		\includegraphics[trim={1cm 7.5cm 1cm 7cm},clip,width=2.5in]{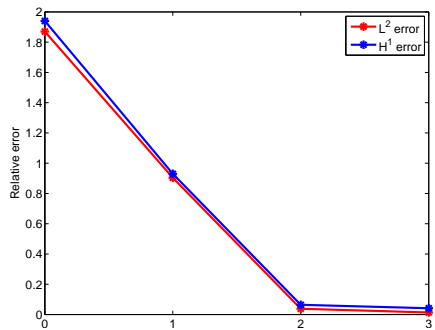}
		\caption{Error for model 4, $H=1/10$, center source.}
	\label{fig:er_m4}
\end{figure}

\section{Conclusion}\label{sec:conclusion}
We demonstrate that the Wavelet-based Edge Multiscale Finite Element
Method (WEMsFEM) for
Helmholtz problems in perforated domains, with possibly large
wavenumbers, are an effective alternative to standard Finite Element
methods with advantages for multiscale problems. Such problems
 have many applications in photonic and phononic crystals and having
 efficient, faster, alternatives to the existing finite element
 approaches is much needed particularly for crystals created from many
 cells and operating at high frequencies.

We have created both the required theory, and error estimates, and
then tested  the convergence analysis and numerical performance of the
algorithm against examples of physical interest. Under the usual
resolution assumption that the product of the coarse-scale mesh size
$H$ and the wavenumber $k$ is bounded above by a certain constant and
the level parameter $\ell$ is sufficiently large, we prove
$\mathcal{O}(H)$ convergence of our methods. Our theoretical results
are supported by extensive 2-d numerical simulations. The success of
this two-dimensional study has motivated further practical tests of
this algorithm for 3-d Helmholtz problems in perforated domains and
these are currently under investigation.

\appendix

\section{Very-weak solutions to the Helmholtz problem}
We establish in this section an a priori estimate, cf. Theorem \ref{lem:very-weak}, which is utilized in the proof to Theorem \ref{thm:proj}. Throughout this section, $\omega_i$ is one coarse neighborhood as defined in \eqref{neighborhood} for all $i=1,2,\cdots,N$.

Let $g\in L^2(\Omega^{\epsilon}\cap\omega_i)$, and $v\in H^{1/2}(\Omega^{\epsilon}\cap\omega_i)$ be the solution to the following problem:
\begin{equation}\label{eq:pde-very}
\left\{
\begin{aligned}
\mathcal{L}_i v&:=\Delta v+k^2 v=0&&\mbox{in }\Omega^{\epsilon}\cap\omega_i,\\
\frac{\partial v}{\partial n}&=0 &&\mbox{ on }\partial Q_1^{\epsilon}\cap \omega_i.\\
 v&=g &&\mbox{on }\partial\omega_i\backslash\partial Q_1^{\epsilon}.
\end{aligned}
\right.
\end{equation}
\begin{align}\label{eq:test-space}
X(\omega_i):=\{z\in H^1(\Omega^{\epsilon}\cap\omega_i):\mathcal{L}_i z\in L^2(\Omega^{\epsilon}\cap\omega_i), \frac{\partial v}{\partial n}=0 \text{ on }\partial Q_1^{\epsilon}\cap \omega_i\text{ and }v=0\text{ on }\partial\omega_i\backslash\partial Q_1^{\epsilon}\}.
\end{align}
This test space $X(\omega_i)$ is endowed with the norm $\|\cdot\|_{X(\omega_i)}$:
\[
\forall z\in X(\omega_i):\|z\|_{X(\omega_i)}^2=\int_{\Omega^{\epsilon}\cap\omega_i}|\nabla z|^2\dx+\|\mathcal{L}_i z\|_{L^2(\Omega^{\epsilon}\cap\omega_i)}^2.
\]
Then we propose the following weak formulation corresponding to Problem \eqref{eq:pde-very}:
seeking $v\in L^2(\Omega^{\epsilon}\cap\omega_i)$ such that
\begin{align}\label{eq:nonstd-variational}
\int_{\Omega^{\epsilon}\cap\omega_i}v\mathcal{L}_i  z\dx=\int_{\partial \omega_i\backslash \partial Q_1^{\epsilon}}g\frac{\partial z}{\partial n}\mathrm{d}s \quad\text{ for all }z\in X(\omega_i).
\end{align}
\begin{theorem}\label{lem:very-weak}
Let the Scale Resolution Assumption \ref{ass:resolution} be valid. Given $g\in L^2(\Omega^{\epsilon}\cap\omega_i)$. Let $v$ be the solution to \eqref{eq:pde-very}. Then there holds
\begin{align*}
\normL{v}{\Omega^{\epsilon}\cap\omega_i}&\leq\Cw \Ce H^{1/2}\|g\|_{L^2(\partial \omega_i\backslash\partial Q_1^{\epsilon})}\\
\normL{\chi_i\nabla v}{\Omega^{\epsilon}\cap\omega_i}&
\leq\Cw \Ce H^{-1/2}\|g\|_{L^2(\partial \omega_i\backslash\partial Q_1^{\epsilon})}.
\end{align*}
\end{theorem}

To prove Theorem \ref{lem:very-weak}, one has to first derive the $L^2(\partial\omega_i\backslash\partial Q_1^{\epsilon})$-estimate of the normal trace $\frac{\partial z}{\partial n}$ for any $z\in X(\omega_i)$. This is established in the following theorem:
\begin{theorem}\label{thm:pw-Regularity}
Let the Resolution Assumption \ref{ass:resolution} be valid. Let $w\in L^2(\Omega^{\epsilon}\cap\omega_i)$ and let $z\in X(\omega_i)$ satisfy
\begin{equation}\label{eq:pde-dual}
\left\{\begin{aligned}
\mathcal{L}_i z:=\Delta z+k^2 z&=w && \text{ in } \Omega^{\epsilon}\cap\omega_i,\\
\frac{\partial z}{\partial n}&=0 &&\mbox{ on }\partial Q_1^{\epsilon}\cap \omega_i.\\
z&=0 &&\text{ on }\partial \omega_i\backslash \partial Q_1^{\epsilon}.
\end{aligned}\right.
\end{equation}
Then there holds
\begin{align*}
\|\frac{\partial z}{\partial n}\|_{L^2(\partial\omega_i\backslash \partial Q_1^{\epsilon})}\leq
& \Cw\Ce H^{1/2}\normL{w}{\Omega^{\epsilon}\cap\omega_i}.
\end{align*}
Here, $\Cw$ is a positive constant independent of the wavenumber $k$ and the mesh size $H$, which can change values among equations.
\end{theorem}
\begin{proof}
An application of the Poincar\'{e} inequality implies
\begin{align}\label{eq:111}
\normL{z}{\Omega^{\epsilon}\cap\omega_i}\leq \Cpoinn{\omega_i}{1/2}H\normL{\nabla z}{\Omega^{\epsilon}\cap\omega_i}.
\end{align}
Testing \eqref{eq:pde-dual} with $z$ and applying the Poincar\'{e} inequality, together with the former result, we obtain
\begin{align*}
\normL{\nabla z}{\Omega^{\epsilon}\cap\omega_i}^2\leq \Cpoin{\omega_i}(Hk)^2\normL{\nabla z}{\Omega^{\epsilon}\cap\omega_i}^2+\Cpoinn{\omega_i}{1/2}H\normL{w}{\Omega^{\epsilon}\cap\omega_i}
\normL{\nabla z}{\Omega^{\epsilon}\cap\omega_i}.
\end{align*}
Thanks to the Resolution Assumption \eqref{ass:resolution}, this yields
\begin{align}\label{eq:gradient}
\normL{\nabla z}{\Omega^{\epsilon}\cap\omega_i}\leq\Ce\Cpoinn{\omega_i}{1/2}H \normL{w}{\Omega^{\epsilon}\cap\omega_i}.
\end{align}
Furthermore, by \eqref{eq:111}, we arrive at
\begin{align}\label{eq:l2}
\normL{ z}{\Omega^{\epsilon}\cap\omega_i}\leq \Ce\Cpoin{\omega_i}H^2\normL{w}{\Omega^{\epsilon}\cap\omega_i}.
\end{align}
One the other hand, a direct calculation results in
\begin{align*}
\normL{\Delta z}{\Omega^{\epsilon}\cap\omega_i}\leq k^2\normL{z}{\Omega^{\epsilon}\cap\omega_i}+\normL{w}{\Omega^{\epsilon}\cap\omega_i}.
\end{align*}
This, together with \eqref{eq:l2} and the Resolution Assumption \eqref{ass:resolution}, yields
\[
\normL{\Delta z}{\Omega^{\epsilon}\cap\omega_i}\leq \Big(\Ce +1\Big)\normL{w}{\Omega^{\epsilon}\cap\omega_i}.
\]
Note that the interface $\partial Q_1^{\epsilon}\cap\omega_i$ has sufficient smoothness, we have the following {\em a priori} estimate
\[
\|z\|_{H^2(\Omega^{\epsilon}\cap\omega_i)}\lesssim \normL{\Delta z}{\Omega^{\epsilon}\cap\omega_i}\leq \Big(\Ce +1\Big)\normL{w}{\Omega^{\epsilon}\cap\omega_i}.
\]
This, together with \eqref{eq:gradient} and applying interpolation between $H^1(\Omega^{\epsilon}\cap\omega_i)$ and $H^2(\Omega^{\epsilon}\cap\omega_i)$ yields the $H^{3/2}(\Omega^{\epsilon}\cap\omega_i)$ regularity estimate
\begin{align}\label{eq:h3/2}
\normHp{z}{\Omega^{\epsilon}\cap\omega_i}{3/2}\lesssim(\Ce+1)\Cpoinn{\omega_i}{1/4}H^{1/2}\normL{w}{\Omega^{\epsilon}\cap\omega_i}.
\end{align}
Since differentiation is continuous from $H^{3/2}(\Omega^{\epsilon}\cap\omega_i)$ to $H^{1/2}(\Omega^{\epsilon}\cap\omega_i)$, by the trace theorem, we have
\begin{align*}
\|\frac{\partial z}{\partial n}\|_{L^2(\partial\omega_i\backslash \partial Q_1^{\epsilon})}&\lesssim \|\nabla z\|_{H^{1/2}(\Omega^{\epsilon}\cap\omega_i)}\lesssim \normHp{z}{\Omega^{\epsilon}\cap\omega_i}{3/2},
\end{align*}
which, together with \eqref{eq:h3/2} and the Resolution Assumption \ref{ass:resolution}, proves the desired assertion.
\end{proof}

\begin{proof}[Proof to Theorem \ref{lem:very-weak}]
The first result can be proved in a similar manner as \cite[Theorem A.1]{GL18}, with the help of Theorem \ref{thm:pw-Regularity}.

To prove the second assertion, recall that $\chi_i$ is the bilinear function supported in $\omega_i$ and $\chi_i=0$ on $\partial\omega_i\backslash \partial Q_1^{\epsilon}$. Multiplying \eqref{eq:pde-very} by $\chi_i^2v$ and applying integration by parts, we arrive at
\begin{align*}
\int_{\Omega^{\epsilon}\cap\omega_i}\chi_i^2|\nabla v|^2\dx=-2 \int_{\Omega^{\epsilon}\cap\omega_i}\nabla v\cdot \nabla\chi_i\chi_i v\dx
+k^2\int_{\Omega^{\epsilon}\cap\omega_i}\chi_i^2 v^2\dx.
\end{align*}
Then an application of the Young's inequality implies
\begin{align*}
\int_{\Omega^{\epsilon}\cap\omega_i}\chi_i^2|\nabla v|^2\dx\leq (4H^{-2}+2k^2) \int_{\Omega^{\epsilon}\cap\omega_i} v^2\dx.
\end{align*}
After taking the square root over the previous estimate, utilizing the first assertion and the Scale Resolution Assumption \ref{ass:resolution}, the second assertion is proved.
\end{proof}

\section{PML for the Helmholtz equation}\label{sec:pml}
To effectively absorb the outgoing wave, we adopt the Perfectly Matched Layer (PML)
\cite{pml1,pml2} in our implementation. Without loss of generality, let the computation domain including the PML areas be $D=(0,1)^2$.
We follow the notations in \cite{ying}, we define
\begin{equation}\label{eq:pml}
d(x)=\left\{
\begin{aligned}
&\frac{C}{\xi}\left(\frac{x-\xi}{\xi}\right)^2,\qquad&&x\in[0,\xi],\\
&0, \qquad&&x\in[\xi,1-\xi],\\
&\frac{C}{\xi}\left(\frac{x-1+\xi}{\xi}\right)^2,\qquad&&x\in[1-\xi,1],
\end{aligned}
\right.
\end{equation}
and
\begin{equation}
g_1(x_1)=\left(1+i\frac{d(x_1)}{k}\right)^{-1},
\end{equation}
and
\begin{equation}
g_2(x_2)=\left(1+i\frac{d(x_2)}{k}\right)^{-1}
\end{equation}
where $x_1$ and $x_2$ are the space variables.
$\xi$ is the thickness of the PML.
The PML method is to replace $\partial_1$ with $g_1(x_1)\partial_1$
and $\partial_2$ with $g_2(x_2)\partial_2$, respectively. Then
then Equation (\ref{eq:model}) becomes
\begin{equation*}
\left(\partial_1\left(\frac{g_1}{g_2}\partial_1\right)+
\partial_2\left(\frac{g_2}{g_1}\partial_2\right)+\frac{k^2}{g_1g_2}\right)u=f(x_1,x_2).
\end{equation*}
In our implementation, we take $C:=100$ and the thickness of the PML $\xi$ equals to one wavelength.

\bibliographystyle{abbrv}
\bibliography{reference}
\end{document}